\documentclass[11pt]{article}
\usepackage{amsmath,amsfonts,amssymb, amsthm,enumerate,graphicx}
\usepackage{wrapfig}
\usepackage{lipsum}
\usepackage{tikz,authblk}
\usepackage{float}
\usepackage{hyperref}
\usepackage{subfig}
\usepackage{url}
\usetikzlibrary{shapes.geometric}
\newtheorem{theorem} {{\textsf{Theorem}}}[section]
\newtheorem{proposition}[theorem]{{\textsf{Proposition}}}
\newtheorem{corollary}[theorem]{{\textsf{Corollary}}}
\newtheorem{conj}[theorem]{{\textsf{Conjecture}}}
\newtheorem{definition}[theorem]{{\textsf{Definition}}}
\newtheorem{remark}[theorem]{{\textsf{Remark}}}

\newtheorem{lemma}[theorem]{{\textsf{Lemma}}}
\newtheorem{question}[theorem]{{\textsf{Question}}}

\begin{document}
\title{A Complete Classification of Discrete $d$-Pseudomanifolds with at Most $2d+7$ Vertices}
\author{Biplab Basak$^{1,2}$, Debolina Ghosh$^2$, and Raju Kumar Gupta$^3$}
\date{}
\maketitle
\vspace{-10mm}
\begin{center}

\noindent {\small Department of Mathematics, IIT Delhi, Hauz Khas, New Delhi 110016$^2$.}

\noindent {\small Mathematical Institute, Slovak Academy of Sciences, Bratislava, 81438$^{3}$.}

\noindent {\small {\em E-mail addresses:} \url{biplab@iitd.ac.in} (B. Basak), 
\url{maz258290@maths.iitd.ac.in} (D. Ghosh) \url{rajukrg3217@gmail.com} (R. K. Gupta).}

\medskip

\date{\today}
\end{center}
\footnotetext[1]{Corresponding author}
\hrule

\begin{abstract}
A simple undirected graph \(M\) is called a discrete \(d\)-pseudomanifold if, for every vertex \(v\), the induced subgraph \(N_M(v)\) on the neighbors of \(v\) is a discrete \((d-1)\)-pseudomanifold, where a discrete \(1\)-pseudomanifold is defined to be an \(n\)-cycle with \(n\geq 4\). 
These objects arise naturally as graph-theoretic analogues of simplicial pseudomanifolds and provide a purely combinatorial framework for studying manifold-like structures through local neighborhood conditions. Understanding discrete pseudomanifolds with a small number of vertices is therefore a fundamental problem in combinatorial topology and extremal graph theory.

In this article, we first prove that every discrete \(d\)-pseudomanifold has at least \(2(d+1)\) vertices. We then provide a complete classification of discrete \(d\)-pseudomanifolds with at most \(2d+6\) vertices by determining all possible combinatorial types of such pseudomanifolds. Further, we establish an equivalence between discrete \(d\)-pseudomanifolds and edge graphs of flag normal \(d\)-pseudomanifolds. As a consequence, we derive a purely combinatorial characterization of flag normal \(d\)-pseudomanifolds with at most \(2d+6\) vertices and prove that each such complex is a simplicial $d$-sphere. 

Finally, we show that this sphere characterization is optimal within the class of flag normal \(d\)-pseudomanifolds by constructing examples on \(2d+7\) vertices that are not spheres. Specifically, we prove that, for $d\geq 3$, every flag normal $d$-pseudomanifold with at most $2d+7$ vertices is either a simplicial $d$-sphere or a flag triangulation of the $(d-2)$-fold suspension of $\mathbb{RP}^{2}$.

\end{abstract}

\noindent {\small {\em MSC 2020\,:}  Primary 05E45; Secondary 05C69, 52B05, 57Q15.
	
}

\section{Introduction}

Graphs are commonly used to model networks, but they can also be viewed as geometric objects. In topology, spaces such as manifolds, and pseudomanifolds are usually described using continuous or simplicial structures. However, in many applications and computational settings, it is desirable to work with discrete objects. This has led to the study of graphs whose local neighborhoods resemble lower-dimensional geometric spaces. Existing graph-theoretic models mainly focus on manifold-like structures, where every point has a nice local neighborhood. However, many important topological spaces contain singularities or other non-manifold features that cannot be adequately captured by these models.

One graph-theoretic discrete model of topological spaces arises in digital topology. In digital topology, the graph-theoretic approach represents a digital image by a graph whose structure is determined by the local adjacency relations of its points. In this setting, digital $n$-dimensional spaces with locally spherical neighborhoods provide graph-theoretic models of continuous $n$-dimensional spaces. As a result, digital $n$-surfaces can be defined and studied as simple undirected graphs (see, for example, Evako \cite{Evako 2} and Daragon et.al. \cite{Daragon}). Several work has been carried out using this graph-theoretic framework to study digital manifolds and  digital spaces with locally non-spherical topologies (see Evako \cite{Evako 1, Evako 3, Evako 4, Evako 5}). 

 In recent years, discrete manifolds have gained significant interest as an independent area of study through the work of Oliver Knill \cite{Knill 1, Knill 2, Knill 3}. Knill defined a discrete \(d\)-manifold \(M\) purely as a finite simple graph where the unit sphere of every vertex is the induced subgraph generated by its neighbors is a discrete \((d-1)\)-sphere. Using this self-contained framework, Knill explored several classic topological themes in a purely graph-theoretic setting, including graph coloring, a discrete analogue of the Morse-Sard theorem, and a variation of the Gauss-Bonnet theorem. Given the rich structural properties of these graphs, expanding these concepts beyond manifold constraints has become a natural next step. For instance, a recent generalization by Basak et al. \cite{BasakDeom} extended the study of graph coloring from discrete manifolds to discrete pseudomanifolds. 

Motivated by these developments, this paper formally introduces the concept of a discrete $d$-pseudomanifold.  A simple undirected graph $M$ is called a discrete $d$-pseudomanifold if, for every vertex $v$, the induced subgraph $N_M(v)$ on the neighbors of $v$ is a discrete $(d - 1)$-pseudomanifold, where a discrete $1$-pseudomanifold is defined to be an $n$-cycle with $n \geq 4$. The structural beauty of this definition lies in its purely local and recursive nature. These objects arise naturally as a graph-theoretic generalization of pseudomanifolds, avoiding the need for complex geometric tracking. They provide a highly accessible, purely combinatorial framework to study manifold-like structures with singularities simply by looking at the local behavior of the vertices. This definition is useful from both theoretical and computational perspectives. Since it depends only on the neighborhoods of vertices, it can be verified using local information, making it suitable for algorithmic study.

In this article, we investigate the structure of discrete $d$-pseudomanifolds and their 
connection with flag normal $d$-pseudomanifolds. We first prove that every discrete 
$d$-pseudomanifold must contain at least $2(d+1)$ vertices, establishing the sharp 
lower bound on the number of vertices. We then provide a complete classification of 
discrete $d$-pseudomanifolds having at most $2d+6$ vertices. This classification gives a 
detailed description of all possible combinatorial types in this range.

Flag simplicial complexes have emerged as an important class of simplicial complexes due to their strong connections with topology, extremal combinatorics, and graph theory. In particular, the study of flag triangulations of manifolds and spheres has led to several fundamental results, including upper bound theorems and structural properties of flag pseudomanifolds, flag homology spheres, and flag spheres; see, for example, \cite{Athanasiadis, LabbeNevo2017, Zheng2017, Zheng2020a}. Moreover, flag spheres have been extensively studied in connection with the Charney–Davis conjecture and  its generalization given by Gal’s conjecture; see \cite{Charney-Davis, Gal}.

In this article, we next relate these two settings. We establish an equivalence between discrete $d$-pseudomanifolds and the 
edge graphs of flag normal $d$-pseudomanifolds. Using this correspondence, we translate 
our classification results into the setting of flag complexes and obtain a purely 
combinatorial characterization of flag normal $d$-pseudomanifolds with at most 
$2d+6$ vertices. In particular, we prove that every flag normal $d$-pseudomanifold in this range is a simplicial $d$-sphere.

Flag spheres have recently attracted significant attention due to their rich 
combinatorial and topological properties. In particular, recent works have studied 
structural aspects of flag spheres, including stable sets and the existence of induced 
equators; see Chudnovsky and Nevo \cite{Nevo23,Nevo20}. Our results contribute to this 
direction by providing a complete classification in the small vertex regime and by 
identifying a precise threshold for when flag normal $d$-pseudomanifolds are forced to 
be spheres.

Finally, we show that the obtained sphere characterization is optimal within the class 
of flag normal $d$-pseudomanifolds. Specifically, we construct examples with 
$2d+7$ vertices that are flag normal $d$-pseudomanifolds but are not simplicial 
$d$-spheres, demonstrating that the bound $2d+6$ cannot be improved. More precisely, we prove that, for every $d\geq 3$, any flag normal $d$-pseudomanifold with at most $2d+7$ vertices is either a simplicial $d$-sphere or a flag triangulation of the $(d-2)$-fold suspension of $\mathbb{RP}^{2}$.
\section{Preliminaries}

Let $M=(V(M),E(M))$ be a simple undirected graph, where $V(M)$ and
$E(M)$ denote the vertex set and edge set of $M$, respectively. For
vertices $u,v\in V(M)$, an edge of $M$ is denoted by $uv$, where
$uv\in E(M)$. For a vertex $v\in V(M)$, the subgraph induced by the set
of neighbors of $v$ is denoted by $N_M(v)$, i.e.,
$V(N_M(v))=\{u\in V(M): uv\in E(M)\},$
and $E(N_M(v))=\{uw: uv, uw, wv\in E(M)\}$. When the ambient graph $M$ is
clear from the context, we simply write $N(v)$.

\begin{definition}
A graph \( M \) is called a discrete \( d \)-pseudomanifold if, for every vertex \( v \), the graph \( N_M(v) \) (or, \( N(v) \) if $M$ is understood) induced on the neighboring vertices of \( v \) is a discrete \( (d-1) \)-pseudomanifold, where a discrete \( 1 \)-pseudomanifold is defined to be an \( n \)-cycle with \( n \geq 4 \).
\end{definition}

By \(M_n^d\) we denote a discrete \(d\)-pseudomanifold with \(n\) vertices. For $n\ge 3$, the cycle graph $C_n$ is the graph with vertex set
$
V(C_n)=\{v_1,v_2,\dots,v_n\}
$
and edge set
$
E(C_n)=\{v_i v_{i+1} : 1\le i\le n-1\}\cup\{v_n v_1\}.
$  Two graphs $M$ and $N$ are said to be \emph{isomorphic}, if there exists a bijection
\[
\phi:V(M)\longrightarrow V(N)
\]
such that
\[
xy\in E(M) \iff \phi(x)\phi(y)\in E(N)
\]
for all $x,y\in V(M)$. The map $\phi$ is called a graph isomorphism.

Let $M$ and $N$ be graphs with disjoint vertex sets. The \emph{join} of
$M$ and $N$, denoted by $M \star N$, is the graph with vertex set
$V(M \star N)=V(M)\cup V(N)$ and edge set
$E(M \star N)=E(M)\cup E(N)\cup\{uv : u\in V(M), v\in V(N)\}$. Thus,
$M \star N$ is obtained from the disjoint union of $M$ and $N$ by adding
all possible edges between every vertex of $M$ and every vertex of $N$. For $d \geq 0$, let $a_0,a_1,\ldots,a_d,b_0,b_1,\ldots,b_d$ be distinct vertices. The graph $\{a_0,b_0\} \star \{a_1,b_1\} \star \cdots \star \{a_d,b_d\}$ is denoted by $S^d_{2d+2}$ and is called the \emph{octahedral $d$-sphere}. By definition, $S^d_{2d+2}$ is a $d$-dimensional discrete pseudomanifold on $2(d+1)$ vertices. 

For a discrete $d$-pseudomanifold $M$ and a vertex $v \in V(M)$, the graph $N_M(v)$ is a discrete $(d-1)$-pseudomanifold.
For an edge $uv \in E(M)$, let $N_M(uv)$ denote the subgraph of $M$ induced by the common neighbors of $u$ and $v$. Thus, $V(N_M(uv))=\{\, w \in V(M)\setminus\{u,v\} : uw, vw\in E(M)\}$ and $E(N_M(uv))=\{\, xy : x,y \in V(N_M(uv)), \, xy \in E(M)\}$. Similarly, if $u_1u_2$, $u_2u_3$, and $u_1u_3$ are edges of $M$, then $N_M(u_1u_2u_3)$ denotes the subgraph of $M$ induced by the common neighbors of $u_1$, $u_2$, and $u_3$.
If $S \subseteq V(M)$, then $M[S]$ denotes the induced subgraph of $M$ on $S$. Thus,
$V(M[S])=S$ and
$E(M[S])=\{\, uv \in E(M) : u,v \in S\,\}$.
For every vertex $v \in V(M)$, the degree of $v$, denoted by $\deg(v)$, is the number of vertices adjacent to $v$. Finally, for any finite set $X$, the notation $|X|$ denotes the cardinality of $X$.
For a graph $M$, we denote by $f_0(M)$ the number of vertices of $M$, i.e.,
$
f_0(M)=|V(M)|,
$
and by $f_1(M)$ the number of edges of $M$, i.e.,
$
f_1(M)=|E(M)|.
$

Let \(M\) be a discrete \(d\)-pseudomanifold. Choose a vertex \(v\in V(M)\). 
Then \(N_M(v)\) is an induced discrete \((d-1)\)-pseudomanifold. Choose a discrete 
\((d-2)\)-pseudomanifold \(W\subset N_M(v)\) such that the induced subcomplex
\[
N_M(v)[X],\qquad X=V(N_M(v))\setminus V(W),
\]
has two disjoint components \(W_1\) and \(W_2\). Define \(Exp(M)\) as the graph with
\[
V(Exp(M))=V(M)\cup\{u,w\}\setminus\{v\},
\]
and edge set
\begin{align*}
E(Exp(M))={}&\big(E(M)\setminus\{vx:x\in V(N_M(v))\}\big)\cup\{ux:x\in V(W_1)\}\\
&\cup\{wx:x\in V(W_2)\}\cup (uw \star W).
\end{align*}
Then \(Exp(M)\) is also a discrete \(d\)-pseudomanifold. We say that \(Exp(M)\) is 
obtained from \(M\) by an edge expansion with respect to the vertex \(v\) and the 
discrete \((d-2)\)-pseudomanifold \(W\). Note that \(Exp(M)\) is uniquely 
determined by the choice of the vertex \(v\) and the discrete \((d-2)\)-pseudomanifold  \(W\). Moreover, \(N_{Exp(M)}(uw)=W\), and the degree of each vertex of \(W\) in 
\(M\) is one less than its degree in \(Exp(M)\). If $W=N_M(vx)$ for some edge $vx\in E(M)$, then $Exp(M)$, obtained from $M$ by an edge expansion with respect to the vertex $v$ and the discrete $(d-2)$-pseudomanifold $W$, is precisely the discrete $d$-pseudomanifold obtained from $M$ by subdividing the edge $vx$.

The reverse operation of an edge expansion is called an \emph{edge contraction}. More
precisely, a graph \(N\) is obtained from a graph \(M\) by an edge contraction of an edge
\(uw\in E(M)\) if \(N\) is obtained by deleting the edge \(uw\) and identifying the
vertices \(u\) and \(w\) into a single vertex \(v\). Although \(M\) is a discrete
\(d\)-pseudomanifold, the graph obtained by an arbitrary edge contraction need not be a
discrete \(d\)-pseudomanifold. However, if \(Exp(M)\) is obtained from \(M\) by an edge
expansion corresponding to the choice of \(v\) and \(W\) as above, then \(M\) is obtained
from \(Exp(M)\) by contracting the edge \(uw\).

\section{Discrete $d$-pseudomanifolds with a few vertices}

\subsection{Discrete $d$-pseudomanifolds on at most $2d+3$ vertices}

\begin{lemma}\label{lemma-1}
For $d\geq 1$, let $M$ be a discrete $d$-pseudomanifold. Then there exist vertices $u,v \in V(M)$ such that the edge $uv \notin E(M)$; that is, $u$ and $v$ are non-adjacent.
\end{lemma}

\begin{proof}

We prove the statement by induction on the dimension $d$.

The result is clear for $d=1$. Indeed, a discrete $1$-pseudomanifold is a cycle $C_n$ with $n\geq 4$, and any such cycle contains two non-adjacent vertices.

Now let $d\geq 2$ and assume that every discrete $(d-1)$-pseudomanifold contains two non-adjacent vertices. Let $M$ be a discrete  $d$-pseudomanifold. Then, for any vertex $x\in V(M)$, the graph $N_M(x)$ is a discrete $(d-1)$-pseudomanifold. By the induction hypothesis, there exist vertices $u,v\in V(N_M(x))$ such that $uv\notin E(N_M(x))$. Since $N_M(x)$ is an induced subgraph of $M$, we also have $uv\notin E(M)$. Thus $u$ and $v$ are non-adjacent vertices of $M$. Hence the result follows by induction.
\end{proof}

\begin{corollary}\label{lem-2}
A discrete $d$-pseudomanifold cannot be a cone. Equivalently, if $M$ is a discrete $d$-pseudomanifold, then there does not exist a graph $G$ and a vertex $v \notin V(G)$ such that $M=G\star\{v\}$.
\end{corollary}

\begin{theorem} \label{vertex edge lowerbound}
For $d\geq 1$, let $M$ be a discrete $d$-pseudomanifold. Then
\begin{enumerate}
    \item[(a)] $f_0(M) \geq 2(d+1)$,
    \item[(b)] $f_1(M) \geq 4\binom{d+1}{2}$.
\end{enumerate}
\end{theorem}

\begin{proof}
We prove part~$(a)$ by induction on the dimension $d$. For $d=1$, a discrete $1$-pseudomanifold is an $n$-cycle for some $n\geq 4$. Hence, $f_0(M)\geq 4$, and the result follows in this case.

Now let $d\ge 2$ and assume that every discrete $(d-1)$-pseudomanifold has at least $2d$ vertices. Let $M$ be a discrete $d$-pseudomanifold. By Lemma~\ref{lemma-1}, there exist non-adjacent vertices $v,w\in V(M)$. Since $\mathrm{lk}_M(v)$ is a discrete $(d-1)$-pseudomanifold, the induction hypothesis gives $f_0(\mathrm{lk}_M(v))\ge 2d$. As $v,w\notin V(\mathrm{lk}_M(v))$, the sets $\{v\}$, $\{w\}$ and $V(\mathrm{lk}_M(v))$ are pairwise disjoint. Therefore,
$
f_0(M)\ge 2+f_0(\mathrm{lk}_M(v))
\ge 2+2d
=2(d+1).
$
This proves part~$(a)$.

We prove part~$(b)$ by induction on the dimension $d$. For $d=1$, a discrete $1$-pseudomanifold is an $n$-cycle for some $n\geq 4$. Hence, $f_1(M)\geq 4$, and the result follows in this case.

Now let $d\ge 2$ and assume that every discrete $(d-1)$-pseudomanifold satisfies
$
f_1(M)\ge 4\binom{d}{2}.
$
Let $M$ be a discrete $d$-pseudomanifold. By Lemma~\ref{lemma-1}, there exist non-adjacent vertices $u,v\in V(M)$. Since the collections of edges containing $u$, containing $v$, and lying in $N_M(u)$ are pairwise disjoint, we have
$
f_1(M)\ge f_0(N_M(u))+f_0(N_M(v))+f_1(N_M(u)).
$
Since $N_M(u)$ and $N_M(v)$ are discrete $(d-1)$-pseudomanifolds, part~(a) and the induction hypothesis imply that
$
f_0(N_M(u)),\,f_0(N_M(v))\ge 2d
$
and
$
f_1(N_M(u))\ge 4\binom{d}{2}.
$
Therefore,
$
f_1(M)\ge 2d+2d+4\binom{d}{2}
=4\binom{d+1}{2}.
$
This proves part~$(b)$.
\end{proof}

\begin{lemma}\label{prop-5}
Let $M$ be a discrete $d$-pseudomanifold and let $\emptyset \subsetneq S \subsetneq V(M)$. If the induced subgraph $M[S]$ is a discrete $k$-pseudomanifold, then $k<d$.
\end{lemma}

\begin{proof}
If possible, let $k\geq d$. Since $M$ is connected, there exist vertices $v_1 \in S$ and $v \in V(M)\setminus S$ such that $v_1v \in E(M)$. Since the induced subgraph $M[S]$ is a discrete $k$-pseudomanifold. Then $N_M(v_1)$ and $N_{M[S]}(v_1)$ are discrete $(d-1)$-pseudomanifold and discrete $(k-1)$-pseudomanifold respectively. Note that $v \in V(N_M(v_1))$ but $v \notin V(N_{M[S]}(v_1))$, so $N_{M[S]}(v_1)$ is a proper induced subgraph of $N_M(v_1)$.

Let $S_1 = S \setminus \{v_1\}$. Then 
$V\big(N_{M[S]}(v_1)\big) 
= \{u \in S_1 : uv_1 \in E(M)\} 
= V\big(N_M(v_1)[S_1]\big)$,
and 
$E\big(N_{M[S]}(v_1)\big) 
= \{uv : u,v \in S_1,\ uv_1, vv_1, uv \in E(M)\} 
= E\big(N_M(v_1)[S_1]\big)$.
Thus, $N_{M[S]}(v_1) = N_M(v_1)[S_1]$, where $\emptyset \subsetneq S_1 \subsetneq V(N_M(v_1))$. Hence, $N_M(v_1)[S_1]$ is a discrete $(k-1)$-pseudomanifold.

Now set $M_1 = N_M(v_1)$. Then  $M_1$ and $M_1[S_1]$ are discrete $(d-1)$-pseudomanifold and  discrete $(k-1)$-pseudomanifold respectively. Repeating this argument, after $d-1$ steps we obtain a discrete $1$-pseudomanifold $M_{d-1}$ and a nonempty proper subset $S_{d-1} \subsetneq V(M_{d-1})$ such that $M_{d-1}[S_{d-1}]$ is also a discrete $(k-d+1)$-pseudomanifold.

However, a discrete $1$-pseudomanifold is a cycle of length at least $4$, and no proper induced subgraph of a cycle is a discrete $1$-pseudomanifold of dimension more than 0. This is a contradiction.

Thus, the result follows.
\end{proof}

 \begin{corollary}\label{corollary:suspension}
Let $M$ be a discrete $d$-pseudomanifold and let $u,v\in V(M)$. If $N_M(u)=N_M(v)$. Then $M=N_M(u) \star \{u,v\}$.
\end{corollary}

\begin{proof}
Since $M$ is a discrete $d$-pseudomanifold, each of 
$N_M(u)$ and $N_M(v)$ is a discrete $(d-1)$-pseudomanifold. 
Since $N_M(u) = N_M(v)$, it follows that $N_M(u) \star \{u,v\}$
is a discrete $d$-pseudomanifold. 

Let $S = V(N_M(u)) \cup \{u, v\}$. Then both $M$ and $M[S]$ are discrete $d$-pseudomanifolds. 
Therefore, by Lemma~\ref{prop-5}, we have $S = V(M)$, i.e.,  $M=N_M(u)  \star \{u,v\}$
\end{proof}

\begin{theorem} \label{prop-4}
Let $M$ be a discrete $d$-pseudomanifold. Assume that for every vertex $v\in V(M)$, the neighbor graph $N_M(v)$ is the $(d-1)$-dimensional octahedral sphere $S^{d-1}_{2d}$. Then $M$ is the $d$-dimensional octahedral sphere $S^{d}_{2d+2}$.
\end{theorem}

\begin{proof}
Let \(v\in V(M)\). Then its neighbor graph \(N_M(v)\) is a
\((d-1)\)-dimensional octahedral sphere, and hence has \(2d\) vertices.
Let \(N_M(v)=S^{d-2}_{2d-2}\star \{w,w_1\}\). Then \(N_M(w)\) is also a
\((d-1)\)-dimensional octahedral sphere, and therefore has \(2d\)
vertices. Thus, in \(M\), the vertex \(w\) is adjacent to two additional
vertices that do not belong to \(N_M(v)\). One of these vertices is
\(v\), and we denote the other one by \(u\), i.e.,
\(N_M(w)=S^{d-2}_{2d-2}\star \{v,u\}\). By construction, \(u\) is not
adjacent to \(v\).

Take any vertex \(x\in S^{d-2}_{2d-2}\). Then
\(N_{N_M(v)}(x)\) is a \((d-2)\)-dimensional octahedral sphere with
\(2d-2\) vertices, i.e.,
\(N_{N_M(v)}(x)=N_{S^{d-2}_{2d-2}}(x)\star\{w,w_1\}\). Further, \(x\) is
adjacent to \(u\) and \(v\), since
\(N_M(w)=S^{d-2}_{2d-2}\star\{v,u\}\). Therefore,
\(N_M(x)=N_{S^{d-2}_{2d-2}}(x)\star\{w,w_1\}\star\{u,v\}\). This implies
that \(w_1u\in E(M)\). Since \(N_M(u)\) is also a
\((d-1)\)-dimensional octahedral sphere, it has \(2d\) vertices. It
follows that \(u\) is adjacent to every vertex of \(S^{d-2}_{2d-2}\) and
to \(w,w_1\). Therefore, \(u\) is adjacent to every vertex in
\(N_M(v)\). By Lemma~\ref{prop-5}, we obtain
\(N_M(u)=N_M(v)\). Consequently, by Corollary~\ref{corollary:suspension},
\(M\) is the suspension of \(N_M(v)\) with suspension vertices
\(\{u,v\}\). Hence, \(M\) is the \(d\)-dimensional octahedral sphere
\(S^d_{2d+2}\).
\end{proof}

The following result follows from Lemma~\ref{lemma-1}, Theorems~\ref{vertex edge lowerbound} and~\ref{prop-4}.

\begin{corollary} \label{2d+2 vertex}
Let $M$ be a discrete $d$-pseudomanifold with $d \geq 0$. Then
$f_0(M)=2d+2$ if and only if $M = S^{d}_{2d+2}$.
\end{corollary}

\begin{theorem} \label{2d+3 vertex}
Let $M$ be a discrete $d$-pseudomanifold with $d \geq 1$. Then
$f_0(M) = 2d+3$ if and only if  $M = C_5 \star S^{d-2}_{2d-2}.
$
\end{theorem}

\begin{proof}
We proceed by induction on the dimension. For $d=1$, if $f_0(M)=5$, then $M$ is a discrete $1$-pseudomanifold with five vertices. Hence $M$ must be a cycle of length five, that is, $M= C_5$.

Now assume that the statement holds for all dimensions strictly smaller than $d$, and let $M$ be a discrete $d$-pseudomanifold with $f_0(M)=2d+3$. For any vertex $v\in V(M)$, the graph $N_M(v)$ is a discrete $(d-1)$-pseudomanifold. Therefore, by Theorem~\ref{vertex edge lowerbound}(a), $f_0(N_M(v))\ge 2d$. Furthermore, if equality holds, then $N_M(v)$ is an octahedral sphere. Suppose that $f_0(N_M(v))=2d$ for every vertex $v\in V(M)$. Then, by Theorem~\ref{prop-4}, the graph $M$ must be the $d$-dimensional octahedral sphere. However, the $d$-dimensional octahedral sphere has $2d+2$ vertices, which contradicts the assumption that $f_0(M)=2d+3$. Consequently, there exists a vertex $u\in V(M)$ such that $f_0(N_M(u))\ge 2d+1$.

By Lemma~\ref{lemma-1}, there exists a vertex $w\in V(M)$ which is not adjacent to $u$. Hence no vertex of $N_M(u)$ can coincide with $u$ or $w$. Since $f_0(M)=2d+3$, it follows that $f_0(N_M(u))=2d+1$, and consequently
$
V(M)=V(N_M(u))\sqcup\{u,w\}.
$
It follows from Lemma~\ref{prop-5} that $N_M(u)=N_M(w)$. Indeed, every vertex of $N_M(u)$ is adjacent to both $u$ and $w$. Hence $M$ is the suspension of $N_M(u)$ with suspension vertices $u$ and $w$.

Since $N_M(u)$ is a discrete $(d-1)$-pseudomanifold on $2d+1=2(d-1)+3$ vertices, the induction hypothesis implies that
$
N_M(u)= C_5\star S^{d-3}_{2d-4}.
$
Therefore,
$
M= \{u,w\}\star\bigl(C_5\star S^{d-3}_{2d-2}\bigr)
= C_5\star S^{d-2}_{2d-2}.
$
This proves the forward implication. The converse is immediate, since
$
f_0\!\left(C_5\star S^{d-2}_{2d-2}\right)=2d+3.
$
\end{proof}

\begin{lemma} \label{lem-3}
 Let $M$ be a discrete $d$-pseudomanifold and let $u,v\in V(M)$ be adjacent vertices. Then $N_M(uv)$ is a discrete $(d-2)$-pseudomanifold.
\end{lemma}
\begin{proof}
By definition,
$
V(N_M(uv))=\{\,w\in V(M):uw,vw\in E(M)\,\}
=\{\,w\in V(N_M(v)):uw\in E(N_M(v))\,\}
=V(N_{N_M(v)}(u)).
$
Hence $N_M(uv)=N_{N_M(v)}(u)$. Since $M$ is a discrete $d$-pseudomanifold, $N_M(v)$ is a discrete $(d-1)$-pseudomanifold. Therefore, $N_{N_M(v)}(u)$ is a discrete $(d-2)$-pseudomanifold, and consequently $N_M(uv)$ is a discrete $(d-2)$-pseudomanifold.
\end{proof}

\begin{lemma} \label{lem-4}
Let $M$ be a discrete $d$-pseudomanifold on $n$ vertices. If there exists a vertex $v\in V(M)$ with $\deg(v)=n-2$, then there exists a vertex $u\in V(M)$ not adjacent to $v$ such that
$
M=M^{d-1}_{n-2}\star\{u,v\},
$
where $M^{d-1}_{n-2}$ is a discrete $(d-1)$-pseudomanifold.

\end{lemma}
\begin{proof}
Since $\deg(v)=n-2$, there exists a unique vertex $u\in V(M)$ that is not adjacent to $v$. Hence
$
V(N_M(v))=V(M)\setminus\{u,v\},
$
and therefore $N_M(u)\subseteq N_M(v)$. Since $M$ is a discrete $d$-pseudomanifold, both $N_M(u)$ and $N_M(v)$ are discrete $(d-1)$-pseudomanifolds. Thus, by Lemma~\ref{prop-5},
$
N_M(u)=N_M(v).
$
Consequently,
$
M=N_M(v)\star\{u,v\},
$
where $N_M(v)$ is a discrete $(d-1)$-pseudomanifold. This completes the proof.
\end{proof}

\begin{lemma}\label{adjacency of vertices}
Let $M$ be a discrete $d$-pseudomanifold. Suppose
$
V(M)=\{v\}\cup V\bigl(N_M(v)\bigr)\cup \{u_1,u_2,\dots,u_k\},
$
where $u_i v\notin E(M)$ for all $i\in\{1,2,\dots,k\}$, and $k\geq 2$. Then, for every
$i\in\{1,2,\dots,k\}$, there exists
$j\in\{1,2,\dots,k\}$ with $j\neq i$ such that
$
u_i u_j\in E(M).
$
\end{lemma}
\begin{proof}
Fix $u_i$. Suppose, for contradiction, that $u_i u_j \notin E(M)$ for all $j \in \{1,2 \dots k\}$ with $j\neq i$. Then $N_M(u_i)\subseteq N_M(v)$. Since both $N_M(u_i)$ and $N_M(v)$ are discrete $(d-1)$-pseudomanifolds, Lemma~\ref{prop-5} implies that
$
N_M(u_i)=N_M(v).
$
Then, by Corollary~\ref{corollary:suspension}, $M = N_M(v)\star \{v,u_i\}$. This is a contradiction.  This completes the proof.
\end{proof}

\begin{lemma} \label{lem-6}
Let $M$ be a discrete $d$-pseudomanifold with $n$ vertices. Then, for every vertex $v \in V(M)$, its degree satisfies
$
\deg(v)\in\{2d,\,2d+1,\,\dots,\,n-2\}.
$
\end{lemma}
\begin{proof}
Since, by definition, $N_M(v)$ is a discrete $(d-1)$-pseudomanifold, it follows from Theorem~\ref{vertex edge lowerbound} that
$
2d \leq \deg(v) \leq n-1.
$
Moreover, if $\deg(v)=n-1$, then $M$ would be a coned graph, which contradicts Lemma~\ref{lem-2}. This completes the proof.
\end{proof}
\subsection{Discrete $d$-pseudomanifold with $2d+4$ vertices}

Let $M$ be a discrete $d$-pseudomanifold with $2d+4$ vertices. By the
definition of a discrete $d$-pseudomanifold, when $d=1$ and
$f_0(M)=6$, we have $M= C_6$. For $d\geq2$, we have the following
results.\\

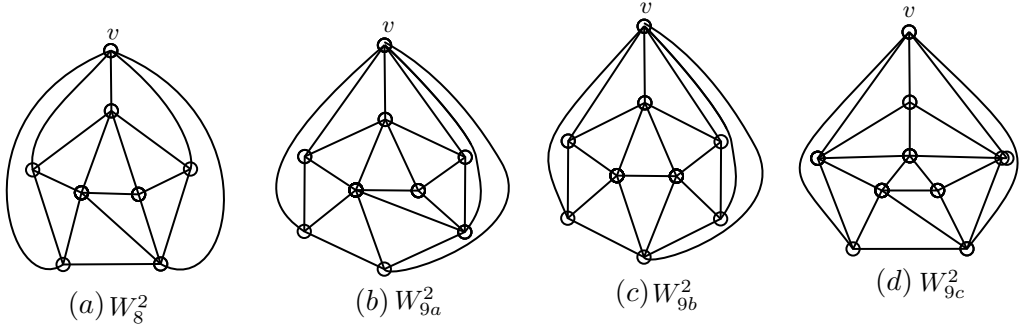
\begin{figure}
    \begin{center}

\tikzset{every picture/.style={line width=0.75pt}} 

\begin{tikzpicture}[x=0.5pt,y=0.5pt,yscale=-1,xscale=1]

\draw   (170.71,172.15) -- (148.26,243.63) -- (74.82,244.02) -- (51.88,172.79) -- (111.14,128.37) -- cycle ;
\draw    (88.65,190) -- (131.63,191.24) ;
\draw    (88.65,190) -- (111.15,128.37) ;
\draw    (88.65,190) -- (51.88,172.79) ;
\draw    (74.81,244.02) -- (88.65,190) ;
\draw    (148.25,243.63) -- (88.65,190) ;
\draw    (131.63,191.24) -- (111.15,128.37) ;
\draw    (170.71,172.15) -- (131.63,191.24) ;
\draw    (148.25,243.63) -- (131.63,191.24) ;
\draw    (110.67,83.02) .. controls (-9.9,130.33) and (34.55,271.69) .. (74.81,244.02) ;
\draw    (110.67,83.02) .. controls (232.9,129.04) and (203.08,274.31) .. (148.25,243.63) ;
\draw    (110.67,83.02) .. controls (166.08,125.57) and (173.83,163.67) .. (170.71,172.15) ;
\draw    (110.67,83.02) .. controls (89.23,109.68) and (46.97,138.24) .. (51.88,172.79) ;
\draw    (110.67,83.02) -- (111.15,128.37) ;
\draw   (317.3,134.83) -- (377.56,163.23) -- (377.22,219.57) -- (316.61,247.5) -- (256.35,219.1) -- (256.7,162.76) -- cycle ;
\draw    (295.18,187.99) -- (342.22,188.42) ;
\draw    (256.69,162.75) -- (295.18,187.99) ;
\draw    (317.3,134.82) -- (295.18,187.99) ;
\draw    (295.18,187.99) -- (256.35,219.1) ;
\draw    (295.18,187.99) -- (316.62,247.51) ;
\draw    (295.18,187.99) -- (377.22,219.58) ;
\draw    (342.22,188.42) -- (377.22,219.58) ;
\draw    (342.22,188.42) -- (377.56,163.23) ;
\draw    (342.22,188.42) -- (317.3,134.82) ;
\draw [color={rgb, 255:red, 0; green, 0; blue, 0 }  ,draw opacity=1 ]   (316.75,78.9) -- (317.3,134.82) ;
\draw [color={rgb, 255:red, 0; green, 0; blue, 0 }  ,draw opacity=1 ]   (316.75,78.9) -- (377.56,163.23) ;
\draw    (316.75,78.9) -- (256.69,162.75) ;
\draw    (316.75,78.9) .. controls (317.82,78.47) and (382.01,141.02) .. (385.72,163.29) .. controls (389.43,185.57) and (394.21,198.42) .. (377.22,219.58) ;
\draw    (316.62,76.59) .. controls (336.38,87.04) and (372.46,111.89) .. (404.68,166.98) .. controls (436.91,222.07) and (321,257.96) .. (316.62,247.51) ;
\draw    (316.75,78.9) .. controls (317.69,78.53) and (367.43,126.83) .. (381.95,153.65) ;
\draw    (256.35,219.1) .. controls (220.33,195.32) and (224.33,158.32) .. (316.75,78.9) ;
\draw   (512.75,122.31) -- (570.13,151.59) -- (569.81,209.66) -- (512.11,238.45) -- (454.73,209.17) -- (455.05,151.11) -- cycle ;
\draw    (491.7,177.11) -- (536.48,177.55) ;
\draw    (455.05,151.1) -- (491.7,177.11) ;
\draw    (512.76,122.31) -- (491.7,177.11) ;
\draw    (491.7,177.11) -- (454.73,209.17) ;
\draw    (491.7,177.11) -- (512.11,238.46) ;
\draw    (536.48,177.55) -- (569.81,209.66) ;
\draw    (536.48,177.55) -- (570.13,151.59) ;
\draw    (536.48,177.55) -- (512.76,122.31) ;
\draw [color={rgb, 255:red, 0; green, 0; blue, 0 }  ,draw opacity=1 ]   (512.24,64.68) -- (512.76,122.31) ;
\draw [color={rgb, 255:red, 0; green, 0; blue, 0 }  ,draw opacity=1 ]   (512.24,64.68) -- (570.13,151.59) ;
\draw    (512.24,64.68) -- (455.05,151.1) ;
\draw    (512.24,64.68) .. controls (513.25,64.24) and (574.36,128.7) .. (577.9,151.65) .. controls (581.44,174.61) and (585.98,187.86) .. (569.81,209.66) ;
\draw    (512.11,62.3) .. controls (530.93,73.07) and (565.27,98.67) .. (595.95,155.45) .. controls (626.63,212.23) and (516.28,249.22) .. (512.11,238.46) ;
\draw    (512.24,64.68) .. controls (513.13,64.29) and (560.49,114.08) .. (574.31,141.71) ;
\draw    (454.73,209.17) .. controls (430.33,153.32) and (428.39,156.07) .. (512.24,64.68) ;
\draw    (536.48,177.55) -- (512.11,238.46) ;
\draw   (781.64,163.95) -- (755.24,232.23) -- (669.49,232.31) -- (642.89,164.09) -- (712.2,121.83) -- cycle ;
\draw   (712.17,162.14) -- (733.17,188.31) -- (691.16,188.31) -- cycle ;
\draw    (712.2,121.83) -- (712.17,162.14) ;
\draw    (642.89,164.09) -- (712.17,162.14) ;
\draw    (642.89,164.09) -- (691.16,188.31) ;
\draw    (691.16,188.31) -- (669.49,232.32) ;
\draw    (691.16,188.31) -- (755.24,232.23) ;
\draw    (733.17,188.31) -- (755.24,232.23) ;
\draw    (733.17,188.31) -- (781.64,163.95) ;
\draw    (712.17,162.14) -- (785.29,164.02) ;
\draw    (711.69,69) -- (712.2,121.83) ;
\draw    (711.69,69) -- (642.89,164.09) ;
\draw    (711.69,69) .. controls (809.82,162.68) and (813.19,159.45) .. (755.24,232.23) ;
\draw    (669.49,232.32) .. controls (614.99,159.45) and (604.41,184.01) .. (711.69,69) ;
\draw    (711.69,69) -- (785.29,164.02) ;

\draw (105.63,64.25) node [anchor=north west][inner sep=0.75pt]  [font=\footnotesize,rotate=-2.96]  {$v$};
\draw (107.35,263.5) node [anchor=north west][inner sep=0.75pt]  [font=\small,rotate=-2.96]  {$W_{8}^{2}$};
\draw (312.63,58.82) node [anchor=north west][inner sep=0.75pt]  [font=\footnotesize]  {$v$};
\draw (322,258.72) node [anchor=north west][inner sep=0.75pt]  [font=\small]  {$W_{9a}^{2}$};
\draw (516,253.72) node [anchor=north west][inner sep=0.75pt]  [font=\small]  {$W_{9b}^{2}$};
\draw (717,245.72) node [anchor=north west][inner sep=0.75pt]  [font=\small]  {$W_{9c}^{2}$};
\draw (76,261.72) node [anchor=north west][inner sep=0.75pt]    {$( a)$};
\draw (292,259.72) node [anchor=north west][inner sep=0.75pt]    {$( b)$};
\draw (488,252.72) node [anchor=north west][inner sep=0.75pt]    {$( c)$};
\draw (684,244.72) node [anchor=north west][inner sep=0.75pt]    {$( d)$};
\draw (506.63,44.82) node [anchor=north west][inner sep=0.75pt]  [font=\footnotesize]  {$v$};
\draw (704.63,47.82) node [anchor=north west][inner sep=0.75pt]  [font=\footnotesize]  {$v$};

\draw   (110.67, 83.02) circle [x radius= 5, y radius= 5]   ;
\draw   (110.67, 83.02) circle [x radius= 5, y radius= 5]   ;
\draw   (110.67, 83.02) circle [x radius= 5, y radius= 5]   ;
\draw   (170.71, 172.15) circle [x radius= 5, y radius= 5]   ;
\draw   (131.63, 191.24) circle [x radius= 5, y radius= 5]   ;
\draw   (131.63, 191.24) circle [x radius= 5, y radius= 5]   ;
\draw   (131.63, 191.24) circle [x radius= 5, y radius= 5]   ;
\draw   (148.25, 243.63) circle [x radius= 5, y radius= 5]   ;
\draw   (88.65, 190) circle [x radius= 5, y radius= 5]   ;
\draw   (88.65, 190) circle [x radius= 5, y radius= 5]   ;
\draw   (88.65, 190) circle [x radius= 5, y radius= 5]   ;
\draw   (88.65, 190) circle [x radius= 5, y radius= 5]   ;
\draw   (148.25, 243.63) circle [x radius= 5, y radius= 5]   ;
\draw   (51.88, 172.79) circle [x radius= 5, y radius= 5]   ;
\draw   (88.65, 190) circle [x radius= 5, y radius= 5]   ;
\draw   (88.65, 190) circle [x radius= 5, y radius= 5]   ;
\draw   (88.65, 190) circle [x radius= 5, y radius= 5]   ;
\draw   (74.81, 244.01) circle [x radius= 5, y radius= 5]   ;
\draw   (88.65, 190) circle [x radius= 5, y radius= 5]   ;
\draw   (88.65, 190) circle [x radius= 5, y radius= 5]   ;
\draw   (295.18, 187.99) circle [x radius= 5, y radius= 5]   ;
\draw   (295.18, 187.99) circle [x radius= 5, y radius= 5]   ;
\draw   (295.18, 187.99) circle [x radius= 5, y radius= 5]   ;
\draw   (295.18, 187.99) circle [x radius= 5, y radius= 5]   ;
\draw   (295.18, 187.99) circle [x radius= 5, y radius= 5]   ;
\draw   (342.22, 188.42) circle [x radius= 5, y radius= 5]   ;
\draw   (342.22, 188.42) circle [x radius= 5, y radius= 5]   ;
\draw   (342.22, 188.42) circle [x radius= 5, y radius= 5]   ;
\draw   (377.21, 219.57) circle [x radius= 5, y radius= 5]   ;
\draw   (295.18, 187.99) circle [x radius= 5, y radius= 5]   ;
\draw   (295.18, 187.99) circle [x radius= 5, y radius= 5]   ;
\draw   (295.18, 187.99) circle [x radius= 5, y radius= 5]   ;
\draw   (295.18, 187.99) circle [x radius= 5, y radius= 5]   ;
\draw   (377.22, 219.58) circle [x radius= 5, y radius= 5]   ;
\draw   (377.22, 219.58) circle [x radius= 5, y radius= 5]   ;
\draw   (377.56, 163.23) circle [x radius= 5, y radius= 5]   ;
\draw   (342.22, 188.42) circle [x radius= 5, y radius= 5]   ;
\draw   (342.22, 188.42) circle [x radius= 5, y radius= 5]   ;
\draw   (377.56, 163.23) circle [x radius= 5, y radius= 5]   ;
\draw   (317.3, 134.83) circle [x radius= 5, y radius= 5]   ;
\draw   (317.3, 134.82) circle [x radius= 5, y radius= 5]   ;
\draw   (342.22, 188.42) circle [x radius= 5, y radius= 5]   ;
\draw   (317.3, 134.82) circle [x radius= 5, y radius= 5]   ;
\draw   (256.71, 162.76) circle [x radius= 5, y radius= 5]   ;
\draw   (295.18, 187.99) circle [x radius= 5, y radius= 5]   ;
\draw   (295.18, 187.99) circle [x radius= 5, y radius= 5]   ;
\draw   (295.18, 187.99) circle [x radius= 5, y radius= 5]   ;
\draw   (256.69, 162.75) circle [x radius= 5, y radius= 5]   ;
\draw   (256.35, 219.1) circle [x radius= 5, y radius= 5]   ;
\draw   (295.18, 187.99) circle [x radius= 5, y radius= 5]   ;
\draw   (295.18, 187.99) circle [x radius= 5, y radius= 5]   ;
\draw   (316.61, 247.5) circle [x radius= 5, y radius= 5]   ;
\draw   (295.18, 187.99) circle [x radius= 5, y radius= 5]   ;
\draw   (316.75, 78.9) circle [x radius= 5, y radius= 5]   ;
\draw   (316.75, 78.9) circle [x radius= 5, y radius= 5]   ;
\draw   (316.75, 78.9) circle [x radius= 5, y radius= 5]   ;
\draw   (316.75, 78.9) circle [x radius= 5, y radius= 5]   ;
\draw   (491.7, 177.11) circle [x radius= 5, y radius= 5]   ;
\draw   (491.7, 177.11) circle [x radius= 5, y radius= 5]   ;
\draw   (491.7, 177.11) circle [x radius= 5, y radius= 5]   ;
\draw   (491.7, 177.11) circle [x radius= 5, y radius= 5]   ;
\draw   (536.48, 177.55) circle [x radius= 5, y radius= 5]   ;
\draw   (536.48, 177.55) circle [x radius= 5, y radius= 5]   ;
\draw   (536.48, 177.55) circle [x radius= 5, y radius= 5]   ;
\draw   (536.48, 177.55) circle [x radius= 5, y radius= 5]   ;
\draw   (536.48, 177.55) circle [x radius= 5, y radius= 5]   ;
\draw   (536.48, 177.55) circle [x radius= 5, y radius= 5]   ;
\draw   (570.13, 151.59) circle [x radius= 5, y radius= 5]   ;
\draw   (536.48, 177.55) circle [x radius= 5, y radius= 5]   ;
\draw   (512.76, 122.32) circle [x radius= 5, y radius= 5]   ;
\draw   (512.76, 122.31) circle [x radius= 5, y radius= 5]   ;
\draw   (536.48, 177.55) circle [x radius= 5, y radius= 5]   ;
\draw   (512.76, 122.31) circle [x radius= 5, y radius= 5]   ;
\draw   (536.48, 177.55) circle [x radius= 5, y radius= 5]   ;
\draw   (455.06, 151.1) circle [x radius= 5, y radius= 5]   ;
\draw   (491.7, 177.11) circle [x radius= 5, y radius= 5]   ;
\draw   (491.7, 177.11) circle [x radius= 5, y radius= 5]   ;
\draw   (491.7, 177.11) circle [x radius= 5, y radius= 5]   ;
\draw   (455.05, 151.1) circle [x radius= 5, y radius= 5]   ;
\draw   (454.73, 209.17) circle [x radius= 5, y radius= 5]   ;
\draw   (491.7, 177.11) circle [x radius= 5, y radius= 5]   ;
\draw   (491.7, 177.11) circle [x radius= 5, y radius= 5]   ;
\draw   (512.1, 238.45) circle [x radius= 5, y radius= 5]   ;
\draw   (491.7, 177.11) circle [x radius= 5, y radius= 5]   ;
\draw   (512.24, 64.68) circle [x radius= 5, y radius= 5]   ;
\draw   (512.24, 64.68) circle [x radius= 5, y radius= 5]   ;
\draw   (512.24, 64.68) circle [x radius= 5, y radius= 5]   ;
\draw   (569.8, 209.66) circle [x radius= 5, y radius= 5]   ;
\draw   (536.48, 177.55) circle [x radius= 5, y radius= 5]   ;
\draw   (712.17, 162.14) circle [x radius= 5, y radius= 5]   ;
\draw   (712.17, 162.14) circle [x radius= 5, y radius= 5]   ;
\draw   (691.16, 188.31) circle [x radius= 5, y radius= 5]   ;
\draw   (691.16, 188.31) circle [x radius= 5, y radius= 5]   ;
\draw   (691.16, 188.31) circle [x radius= 5, y radius= 5]   ;
\draw   (691.16, 188.31) circle [x radius= 5, y radius= 5]   ;
\draw   (691.16, 188.31) circle [x radius= 5, y radius= 5]   ;
\draw   (691.16, 188.31) circle [x radius= 5, y radius= 5]   ;
\draw   (733.17, 188.31) circle [x radius= 5, y radius= 5]   ;
\draw   (733.17, 188.31) circle [x radius= 5, y radius= 5]   ;
\draw   (733.17, 188.31) circle [x radius= 5, y radius= 5]   ;
\draw   (733.17, 188.31) circle [x radius= 5, y radius= 5]   ;
\draw   (712.2, 121.84) circle [x radius= 5, y radius= 5]   ;
\draw   (712.17, 162.14) circle [x radius= 5, y radius= 5]   ;
\draw   (712.17, 162.14) circle [x radius= 5, y radius= 5]   ;
\draw   (712.2, 121.83) circle [x radius= 5, y radius= 5]   ;
\draw   (781.6, 163.92) circle [x radius= 5, y radius= 5]   ;
\draw   (712.17, 162.14) circle [x radius= 5, y radius= 5]   ;
\draw   (785.29, 164.02) circle [x radius= 5, y radius= 5]   ;
\draw   (755.24, 232.23) circle [x radius= 5, y radius= 5]   ;
\draw   (755.24, 232.23) circle [x radius= 5, y radius= 5]   ;
\draw   (733.17, 188.31) circle [x radius= 5, y radius= 5]   ;
\draw   (669.49, 232.31) circle [x radius= 5, y radius= 5]   ;
\draw   (691.16, 188.31) circle [x radius= 5, y radius= 5]   ;
\draw   (691.16, 188.31) circle [x radius= 5, y radius= 5]   ;
\draw   (642.89, 164.09) circle [x radius= 5, y radius= 5]   ;
\draw   (642.89, 164.09) circle [x radius= 5, y radius= 5]   ;
\draw   (691.16, 188.31) circle [x radius= 5, y radius= 5]   ;
\draw   (642.89, 164.09) circle [x radius= 5, y radius= 5]   ;
\draw   (642.89, 164.09) circle [x radius= 5, y radius= 5]   ;
\draw   (642.89, 164.09) circle [x radius= 5, y radius= 5]   ;
\draw   (642.89, 164.09) circle [x radius= 5, y radius= 5]   ;
\draw   (711.69, 69) circle [x radius= 5, y radius= 5]   ;
\draw   (711.69, 69) circle [x radius= 5, y radius= 5]   ;
\draw   (711.69, 69) circle [x radius= 5, y radius= 5]   ;
\draw   (111.15, 128.38) circle [x radius= 5, y radius= 5]   ;
\draw   (131.63, 191.24) circle [x radius= 5, y radius= 5]   ;
\draw   (111.15, 128.37) circle [x radius= 5, y radius= 5]   ;
\draw   (131.63, 191.24) circle [x radius= 5, y radius= 5]   ;
\end{tikzpicture}

\caption{Some discrete $2$-pseudomanifolds on $8$ and $9$ vertices.}
\label{fig:1}
    \end{center}
\end{figure}

\begin{theorem}\label{dim-2,vertex-8}
Let $M$ be a discrete $2$-pseudomanifold with $8$ vertices. Then, up to isomorphism,
$
M = C_6 \star S^{0}_2 \quad \text{or} \quad M = W_8^{2}$, where $W_8^2$ is the graph shown in Figure \ref{fig:1}.
\end{theorem}

\begin{proof}
Since $M$ is a discrete $2$-pseudomanifold on $8$ vertices, Lemma~\ref{lem-6} implies that
$
\deg(v)\in\{4,5,6\}
$
for all $v\in V(M)$.

\noindent \textbf{Case 1.} Suppose there exists $v\in V(M)$ with $\deg(v)=6$. Then, by Lemma~\ref{lem-4},
$
M= C_6 \star S_2^0.
$

\noindent \textbf{Case 2.} Suppose that $\deg(x)\leq 5$ for all $x\in V(M)$ and that there exists
$v\in V(M)$ with $\deg(v)=5$. By Lemma~\ref{adjacency of vertices}, there
exist vertices $u,w\in V(M)\setminus V(N_M(v))$ such that $uw\in E(M)$.
Consequently, one of $N_M(u)$ and $N_M(w)$ is a 5-cycle, while
the other is a 4-cycle. Hence, up to isomorphism,
$
M= W_8^2.
$

\noindent \textbf{Case 3.} If $\deg(v)=4$ for all $v\in V(M)$, then by Theorem~\ref{prop-4},
$M= S^2_6$, and hence $f_0(M)=6$, contradicting $f_0(M)=8$.
Therefore, this case does not occur.
\end{proof}

 \begin{theorem}\label{3dim with 10 vertices}
Let $M$ be a discrete $3$-pseudomanifold with $10$ vertices. Then, up to isomorphism,
$
M = W_8^{2} \star S^{0}_2$ or $M = C_6 \star C_4$ or $M = C_5 \star C_5.$

\end{theorem}
\begin{proof}
    $M$ is a discrete $3$-pseudomanifold with $10$ vertices.
By Lemma~\ref{lem-6}, the degree of any vertex $v \in V(M)$ satisfies
$
\deg(v) \in \{6,7,8\}.
$

\noindent \textbf{Case 1:} Suppose there exists a vertex $v \in V(M)$ such that $\deg(v)=8$.  
Then, by Lemma~\ref{lem-4}, we have
$
M = M^2_8 \star S^{0}_2,
$
where $M^2_8$ is a discrete $2$-pseudomanifold with $8$ vertices. By the  Theorem~\ref{dim-2,vertex-8}, it follows that
$
M = C_6 \star S_2^{0} \star S_2^{0} = C_6 \star C_4 $ or $ M = W_8^{2} \star S_2^{0}.
$

\noindent \textbf{Case 2:} Suppose $\deg(x)\leq 7$ for all $x\in V(M)$ and there exists a vertex $v\in V(M)$ with $\deg(v)=7$. Then $N_M(v)$ is a discrete $2$-pseudomanifold on $7$ vertices. By Theorem~\ref{2d+3 vertex},
$
N_M(v)= C_5\star S_2^0=C_5\star\{a,b\}
$
for some vertices $a,b\in V(M)$.

By Lemma~\ref{adjacency of vertices}, there exist vertices $u,w\in V(M)\setminus V(N_M(v))$ such that $uw\in E(M)$. 
Since $a$ and $b$ are adjacent to $v$ and have degree $5$ in $N_M(v)$, if either belongs to $N_M(uw)$, then its degree in $M$ would be $8$, contradicting $\deg(x)\leq 7$ for all $x\in V(M)$. Thus neither $a$ nor $b$ is adjacent to both $u$ and $w$ and so $N_M(uw)\subseteq C_5$. By Lemma~\ref{lem-3}, $N_M(uw)$ is a discrete $1$-pseudomanifold and hence a cycle $C_n$, $n\geq 4$. Therefore, by Lemma~\ref{prop-5},
$
N_M(uw)= C_5.
$

If $au,aw\notin E(M)$, then $N_M(a)= C_5\star v$, contradicting Corollary~\ref{lem-2}. Hence either $au\in E(M)$ or $aw\in E(M)$. Similarly, either $bu\in E(M)$ or $bw\in E(M)$. Since neither $a$ nor $b$ can be adjacent to both $u$ and $w$, without loss of generality, assume $au\in E(M)$ and $bw\in E(M)$. It follows that $M$ is the join of two $5$-cycles,
$
M= C_5\star C_5,
$
where one of the cycles is $C_5(v,a,u,w,b)$. This completes the proof.

\noindent \textbf{Case 3:} If $\deg(v)=6$ for all $v\in V(M)$, then by Theorem~\ref{prop-4},
$M= S^3_8$, and hence $f_0(M)=8$, contradicting $f_0(M)=10$.
Therefore, this case does not occur.
\end{proof}

\begin{theorem} \label{2d+4,d>=4}
Let $M$ be a discrete $d$-pseudomanifold with $2d+4$ vertices and $d \geq 4$. Then, up to isomorphism,
$
M = M^{d-1}_{2d+2} \star S^{0}_2,
$  where  $M^{d-1}_{2d+2}$ is a discrete $(d-1)$-pseudomanifold.
\end{theorem}
\begin{proof}
Let $M$ be a discrete $d$-pseudomanifold with $d \geq 4$ with
$
f_0(M) = 2d+4.
$
By Lemma~\ref{lem-6}, the degree of any vertex $v \in V(M)$ satisfies
$
\deg(v) \in \{2d,\, 2d+1,\, 2d+2\}.
$

\noindent  \textbf{Case 1:} Suppose there exists a vertex $v \in V(M)$ such that $\deg(v)=2d+2$.  
Then, by Lemma~\ref{lem-4}, $M$ decomposes as
$
M = M^{d-1}_{2d+2} \star S^{0}_2,
.$

\noindent \textbf{Case 2: }Suppose $\deg(x)\leq 2d+1$ for all $x\in V(M)$ and let $v\in V(M)$ with
$\deg(v)=2d+1$. Then, by Theorem~\ref{2d+3 vertex},
$N_M(v)= C_5\star S^{d-3}_{2d-4}$. By Lemma~\ref{adjacency of vertices},
there exist $u,w\in V(M)\setminus V(N_M(v))$ such that $uw\in E(M)$. For
any $x\in V(S^{d-3}_{2d-4})$, we have $\deg_{N_M(v)}(x)=2d-1$. Since
$C_5\subseteq N_M(x)$ and $S^{d-1}_{2d}$ contains no induced $5$-cycle,
$N_M(x)\not= S^{d-1}_{2d}$, and hence $\deg(x)\geq 2d+1$. Thus,
$\deg(x)=2d+1$. As $x$ is adjacent to $2d-1$ vertices of $N_M(v)$ and to
$v$, it can be adjacent to exactly one of $u,w$; otherwise
$\deg(x)=2d+2$, contradicting the assumption. On the other hand, by
Lemma~\ref{lem-3}, $N_M(uw)$ is a discrete $(d-2)$-pseudomanifold, and
Theorem~\ref{vertex edge lowerbound} implies
$f_0(N_M(uw))\geq 2(d-1)\geq 6$. Hence, some vertex
$x\in V(S^{d-3}_{2d-4})$ must be adjacent to both $u$ and $w$, which is a
contradiction. Therefore, this case is impossible.

\noindent \textbf{Case 3:} If $\deg(v)=2d$ for all $v\in V(M)$, then by Theorem~\ref{prop-4},
$M= S^d_{2d+2}$, implying $f_0(M)=2d+2$, a contradiction to
$f_0(M)=2d+4$. Hence, this case is impossible.

This completes the proof.
\end{proof}

\subsection{Discrete $d$-pseudomanifold with $2d+5$ vertices}

Let $M$ be a discrete $d$-pseudomanifold with $2d+5$ vertices. By the
definition of a discrete $d$-pseudomanifold, when $d=1$ and
$f_0(M)=7$, we have $M= C_7$. For $d\geq2$, we have the following
results.

\begin{theorem}\label{2dim 9 vertex}
Let $M$ be a discrete $2$-pseudomanifold with $9$ vertices.
Then, up to isomorphism,
$
M= C_7\star S_2^0$ or $
 W^2_{9a}$ or $ W^2_{9b}$ or $W^2_{9c},
$
where $W^2_{9a}$, $W^2_{9b}$, and $W^2_{9c}$ are the graphs depicted
in Figure~\ref{fig:1}.
\end{theorem}

\begin{proof}
Let $M$ be a discrete $2$-pseudomanifold. By Lemma~\ref{lem-6}, $\deg(v)\in\{4,5,6,7\}$ for all $v\in V(M)$. 

\noindent \textbf{Case 1:} Suppose there exists $v\in V(M)$ with $\deg(v)=7$. Then, by Lemma~\ref{lem-4}, $M= C_7*S^0_2$.

\noindent \textbf{Case 2:} Suppose $\deg(x)\leq6$ for all $x\in V(M)$ and let $\deg(v)=6$. Then
$N_M(v)= C_6$. Let $u,w\notin V(N_M(v))$ be the remaining vertices.
By Lemma~\ref{adjacency of vertices}, $uw\in E(M)$. Since
$N_M(uw)= S^0_2$, exactly two vertices of $C_6$ are adjacent to both
$u$ and $w$. Hence $\deg(u)+\deg(w)=10$, and this determines two possible
structures for $M$, namely $W^2_{9a}$ and $W^2_{9b}$.

\noindent  \textbf{Case 3:} 
Suppose $\deg(x)\leq5$ for all $x\in V(M)$ and let $\deg(v)=5$. Then
$N_M(v)= C_5$. Let $u_1,u_2,u_3\in V(M)\setminus V(N_M(v))$ be the
remaining vertices. By Lemma~\ref{adjacency of vertices}, at least two of
$u_1u_2,u_2u_3,u_1u_3$ are edges. We claim that all three are edges.
Otherwise, assume $u_1u_2,u_2u_3\in E(M)$ but $u_1u_3\notin E(M)$. Then
$V(N_M(u_1u_2))$ and $V(N_M(u_2u_3))$ are contained in $V(C_5)$. If they
have a common vertex, then its degree is at least $6$, a contradiction.
Hence, they are disjoint, which implies $\deg(u_2)\geq6$, again a
contradiction. Therefore,
$u_1u_2,u_2u_3,u_1u_3\in E(M)$.

Consequently, two of $N_M(u_1),N_M(u_2),N_M(u_3)$ are $5$-cycles and the
remaining one is a $4$-cycle. Hence, up to isomorphism, the only possible
structure is $M= W^2_{9c}$.

\noindent   \textbf{Case 4:} Suppose that $\deg(v) = 4$ for all $v \in V(M)$. Then, by Proposition~\ref{prop-4}, $M= S^2_6$. This contradicts the assumption that $f_0(M) = 9$. Therefore, this case cannot occur.
\end{proof}

\begin{figure}[ht]
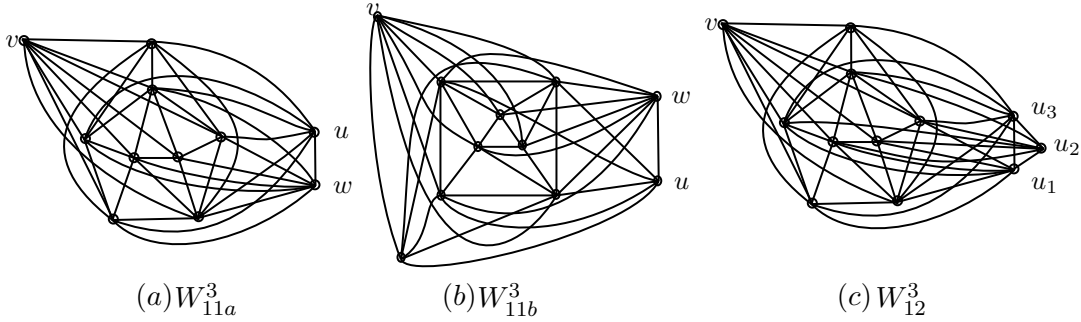

   \begin{center}

\tikzset{every picture/.style={line width=0.75pt}} 

     

\caption{Some discrete $3$-pseudomanifolds $W^3_{11a}$, $W^3_{11b}$ and $W_{12}^3$ with $11$ and $12$ vertices}
\label{fig:3}
   \end{center}
\end{figure}

\begin{theorem}\label{3 dim 11 vertex}
Let $M$ be a discrete $3$-pseudomanifold with $11$ vertices. Then, up to
isomorphism,
$M= M^{2}_{9}\star S^0_2$, where $M^{2}_{9}$ is a discrete $2$-pseudomanifold, or
$M= C_6\star C_5$, or $M= W^3_{11a}$, or
$M= W^3_{11b}$, where $W^3_{11a}$ and $W^3_{11b}$ are depicted in
Figure~\ref{fig:3}.
\end{theorem}

\begin{proof}
 Let $M$ be a discrete $3$-pseudomanifold. Then for all $v \in V(M)$, by Lemma~\ref{lem-6}, we have $\deg(v) \in \{6,7,8,9\}.$

\noindent \textbf{Case 1:} Suppose there exists a vertex $v \in V(M)$ such that 
$\deg(v) = 9$. Then, by Lemma~\ref{lem-4}, 
$M = M^{2}_{9} \star S^0_2.$

\noindent  \textbf{Case 2:} 
Suppose $\deg(x)\leq8$ for all $x\in V(M)$ and let
$v\in V(M)$ with $\deg(v)=8$. Then $N_M(v)$ is a discrete $2$-pseudomanifold on $8$ vertices, and by Theorem~\ref{dim-2,vertex-8},
$N_M(v)= C_6\star \{a,b\}$ or $N_M(v)= W_8^2$, where $a, b$ are some vertices. Since
$f_0(M)=11$, there exist $u,w\notin V(N_M(v))\cup\{v\}$ with
$V(M)=V(N_M(v))\cup\{u,w,v\}$. By Lemma~\ref{adjacency of vertices},
$uw\in E(M)$.

First, assume $N_M(v)= C_6\star \{a,b\}$. Neither $a$ nor $b$ can
be adjacent to both $u$ and $w$, otherwise its degree would be $9$. Hence
$N_M(uw)\subseteq C_6$. By Lemmas~\ref{prop-5} and \ref{lem-3}, $N_M(uw)= C_6$. Thus, $u$ and $w$ are adjacent
to all vertices of $C_6$. Also, $a$ (resp. $b$) must be adjacent to at
least one of $u,w$, otherwise $N_M(a)= v\star C_6$ (resp.
$N_M(b)= v\star C_6$), contradicting Lemma~\ref{lem-2}. Without loss of generality, assume $au,bw\in E(M)$. Then $C_5(v,a,u,w,b)$ is a $5$-cycle and
every vertex of $C_6$ is adjacent to all its vertices. Therefore,
$M= C_6\star C_5$.

Now assume $N_M(v)= W_8^2$. Since $uw\in E(M)$, Lemma~\ref{lem-3}
implies that $N_M(uw)= C_n$ for some $n\geq4$. Moreover,
$N_M(uw)\subseteq N_M(v)= W_8^2$. Since $W_8^2$ contains only induced
$4$- and $5$-cycles, we have $N_M(uw)= C_4$ or $C_5$.

Let $N_M(uw)$ be a 5-cycle, say $C_5$. Then this cycle is the neighbor of a vertex $p$ in
$W_8^2$. Without loss of generality, let $p$ be adjacent to $u$. Then
$N_M(u)=C_5\star\{p,w\}$. Consequently, $w$ is adjacent to the vertices of
$C_5$, to $u$, and to the remaining two vertices of
$V(N_M(v))\setminus(V(C_5)\cup\{p\})$. Hence $N_M(w)= W_8^2$, and
therefore $M= W^3_{11a}$.

Let $N_M(uw)$ be a 4-cycle, say $C_4$. Then this cycle is the neighbor of a vertex $q$ in
$W_8^2$. Taking $q$ adjacent to $u$, we get
$N_M(u)=C_4\star\{q,w\}$. In this case, $w$ is adjacent to the vertices of
$C_4$, to $u$, and to the remaining three vertices of
$V(N_M(v))\setminus(V(C_4)\cup\{q\})$. Thus $N_M(w)= W_8^2$, and
$M= W^3_{11b}$.

\noindent \textbf{Case 3: } Suppose $\deg(x)\leq7$ for all $x\in V(M)$ and let $\deg(v)=7$. Then
$N_M(v)= C_5\star\{a,b\}$ by Theorem~\ref{2d+3 vertex}. Since
$f_0(M)=11$, there exist $u_1,u_2,u_3\notin V(N_M(v))\cup\{v\}$ such that
$V(M)=\{v\}\cup V(N_M(v))\cup\{u_1,u_2,u_3\}$. By
Lemma~\ref{adjacency of vertices}, we may assume
$u_1u_2,u_2u_3\in E(M)$. Since $\deg(a),\deg(b)\leq7$, neither $a$ nor $b$
belongs to $N_M(u_1u_2)$ or $N_M(u_2u_3)$. Hence each of these links
contains at least four vertices of $C_5$, and therefore they share a
vertex $x\in C_5$. This gives $\deg(x)\geq8$, contradicting the degree
assumption. Thus, this case cannot occur.

\noindent \textbf{Case 4:} Suppose $\deg(v)=6$ for all $v\in V(M)$. By Proposition~\ref{prop-4},
$M= S_8^3$, contradicting $f_0(M)=11$. Hence, this case is impossible.
\end{proof}

\begin{theorem}\label{4dim with 13 vertices}
Let $M$ be a discrete $4$-pseudomanifold with $13$ vertices.
Then, up to isomorphism, $M= M^{3}_{11}\star S^0_2$, where $M^{3}_{11}$ is a discrete $3$-dimensional pseudomanifold, or $M= W_8^2\star C_5$.
\end{theorem}

\begin{proof}
Let $M$ be a discrete $4$-pseudomanifold. By Lemma~\ref{lem-6},
$\deg(v)\in\{8,9,10,11\}$ for all $v\in V(M)$.

\noindent \textbf{Case 1:} Suppose there exists $v\in V(M)$ with $\deg(v)=11$. Then,
by Lemma~\ref{lem-4},
$M= M^{3}_{11}\star S^0_2$.

\noindent \textbf{Case 2:} Suppose $\deg(x)\leq10$ for all $x\in V(M)$ and let
$\deg(v)=10$. Then $N_M(v)$ is a $3$-dimensional discrete
pseudomanifold on $10$ vertices. By Theorem~\ref{3dim with 10 vertices},
$
N_M(v)= W_8^2\star S_2^0$ or $C_6\star C_4$ or $ C_5\star C_5 .
$
Since $f_0(M)=13$, there exist $u,w\notin V(N_M(v))\cup\{v\}$ with
$V(M)=V(N_M(v))\cup\{u,v,w\}$. By Lemma~\ref{adjacency of vertices},
$uw\in E(M)$.

First, assume $N_M(v)= C_6\star C_4$. For every $x\in V(C_4)$,
$\deg_{N_M(v)}(x)=8$, and $xv\in E(M)$. Hence $x$
cannot be adjacent to both $u$ and $w$, implying
$N_M(uw)\subseteq C_6$, which is impossible by Lemma \ref{prop-5}. Therefore,
$N_M(v)\neq C_6\star C_4$.

Now suppose $N_M(v)= C_5\star C_5$. Since $\deg(u)\leq10$ and
$\deg(uw)\leq8$, the possible links of $uw$ are
$C_4\star S_2^0$, $C_5\star S_2^0$, $C_6\star S_2^0$, or $W_8^2$. Only
$C_5\star S_2^0$ can occur as a subcomplex of $C_5\star C_5$. Hence
$N_M(uw)= C_5\star S_2^0$. Let $N_M(v)=C_5\star C_5(a,b,c,d,e)$. Then, up to isomorphism,
$N_M(uw)=C_5\star \{a, c\}$. Consequently, up to isomorphism,
\[
M=
v\star(C_5\star C_5(a,b,c,d,e))
\cup
u\star(C_5\star\{ab,bc\})
\cup
w\star(C_5\star\{cd,de,ae\})\cup\{uw\}.
\]
This graph is isomorphic to $C_5\star W_8^2$.

Finally, let $N_M(v)= W_8^2\star \{a, b\}$. If $a$ or $b$ is
adjacent to both $u$ and $w$, then its degree is at least $11$, a
contradiction. Hence $N_M(uw)\subseteq W_8^2$. Since $N_M(uw)$ is a
$2$-dimensional discrete pseudomanifold, $N_M(uw)= W_8^2$.

Moreover, $u$ (respectively, $w$) must be adjacent to at least one of $a,b$; otherwise,
$N_M(u)= w\star W_8^2$, contradicting Corollary \ref{lem-2}.
Thus, without loss of generality, assume $au,bw\in E(M)$. Therefore every vertex of
$W_8^2$ is adjacent to $v,a,u,w,b$, and hence
$M= W_8^2\star C_5$, where $C_5=C_5(v,a,u,w,b)$.

\noindent \textbf{Case 3:} Suppose $\deg(x)\leq9$ for all $x\in V(M)$ and let $\deg(v)=9$. Since
$f_0(M)=13$, there exist $u_1,u_2,u_3\notin V(N_M(v))\cup\{v\}$ with
$V(M)=V(N_M(v))\cup\{v,u_1,u_2,u_3\}$. By
Lemma~\ref{adjacency of vertices}, we may assume
$u_1u_2,u_2u_3\in E(M)$.

As $\deg(v)=9$, by Theorem~\ref{2d+3 vertex},
$N_M(v)= C_5\star S^1_4$. For every $x\in V(S^1_4)$,
$\deg_{N_M(v)}(x)=7$, and since $xv\in E(M)$, $x$ can be adjacent to at
most one of $u_1,u_2,u_3$.

If $f_0(N_M(u_1u_2))=6$ and $N_M(u_1u_2)$ does not contain any vertex of $V(S^1_4)$, then $N_M(u_1u_2)$ contains a $C_5$. However, this is not possible, since $f_0(N_M(u_1u_2))=6$ implies that $N_M(u_1u_2)$ is an octahedral $2$-sphere. Hence, there exists some $x\in V(S^1_4)$ such that $x\in N_M(u_1u_2)$. On the other hand, if $f_0(N_M(u_1u_2))\geq 7$, then again some $x\in V(S^1_4)$ belongs to $N_M(u_1u_2)$. This implies that $x$ is adjacent to both $u_1$ and $u_2$, which is a contradiction. Therefore, this case cannot occur.
 
\noindent \textbf{Case 4:} If $\deg(v)=8$ for all $v\in V(M)$, then by Proposition~\ref{prop-4}, $M= S^4_{10}$, contradicting $f_0(M)=13$; hence this case is impossible.
\end{proof}

\begin{theorem}\label{2d+5,d>=5}
Let $M$ be a discrete $d$-pseudomanifold with $2d+5$ vertices, where
$d\geq5$. Then, up to isomorphism,
$M= M^{d-1}_{2d+3}\star S^0_2$, or,
when $d=5$, the additional possibility
$M= C_5\star C_5\star C_5$ occurs.
\end{theorem}

\begin{proof}
Let $M$ be a discrete  $d$-pseudomanifold with $d \geq 5$ and
$f_0(M)=2d+5$. By Lemma~\ref{lem-6}, the degree of any vertex
$v\in V(M)$ satisfies $\deg(v)\in\{2d+3,\,2d+2,\,2d+1,\,2d\}$.

\noindent \textbf{Case 1:} Suppose that there exists a vertex $v\in V(M)$ such that $\deg(v)=2d+3$. Then, by Lemma~\ref{lem-4},
$M= M^{d-1}_{2d+3}\star S^0_2$.

\noindent \textbf{Case 2:} Suppose that $\deg(v)\leq 2d+2$ for all $v\in V(M)$, and that there exists a
vertex $v\in V(M)$ with $\deg(v)=2d+2$. Then there exist vertices
$u,w\in V(M)\setminus V(N_M(v))$ such that
$\{v\}\cup\{u,w\}\cup V(N_M(v))=V(M)$. By Lemma~\ref{adjacency of vertices},
we have $uw\in E(M)$.

Since $N_M(v)$ is a discrete $(d-1)$-pseudomanifold with $2d+2$ vertices,
Theorem~\ref{2d+4,d>=4} gives
$N_M(v)= M^3_{10}\star S^{d-5}_{2d-8}$,
where $M^3_{10}$ is a discrete $3$-pseudomanifold with $10$ vertices.

For $x\in S^{d-5}_{2d-8}$, we have
$\deg_{N_M(v)}(x)=2d$, and since $xv\in E(M)$, $x$ cannot be adjacent to
both $u$ and $w$, otherwise $\deg(x)=2d+3$, a contradiction. Hence no vertex
of $S^{d-5}_{2d-8}$ lies in $N_M(uw)$. Therefore,
$N_M(uw)\subseteq M^3_{10}$. This implies $d=5$, and by Lemma~\ref{prop-5},
$N_M(uw)=M^3_{10}$. By Theorem~\ref{3dim with 10 vertices},
$
M_{10}^3= W_8^2\star S_2^0,
C_6\star C_4, \text{ or } 
C_5\star C_5.
$

If $M_{10}^3= C_6\star C_4$, then
$N_M(v)= C_6\star C_4\star S^0_2$. Hence, for any
$x\in V(C_4)$, $\deg_{N_M(v)}(x)=10$. Since
$xv,xu,xw\in E(M)$, we obtain $\deg(x)\geq 13$, a contradiction.

If $M_{10}^3= W_8^2\star S_2^0$, then
$N_M(v)= W_8^2\star S_2^0\star S_2^0
= W_8^2\star S_4^1$.
Let $x\in V(S_4^1)$. Then $\deg_{N_M(v)}(x)=10$, and since $xv\in E(M)$,
$x$ cannot be adjacent to both $u$ and $w$, otherwise
$\deg(x)\geq13$, a contradiction. Hence
$N_M(uw)\subseteq W_8^2$, contradicting Lemma~\ref{prop-5}.

Therefore, $M_{10}^3= C_5\star C_5$. Let
$N_M(v)=C_5\star C_5\star\{a,b\}$. If $a$ (respectively, $b$) is not
adjacent to both $u$ and $w$, then
$N_M(a)=v\star C_5\star C_5$ (respectively,
$N_M(b)=v\star C_5\star C_5$), which contradicts Corollary~\ref{lem-2}.
Without loss of generality, assume that $au,bw\in E(M)$.
Since
$N_M(uw)= C_5\star C_5$ and
$N_M(v)= C_5\star C_5\star S_2^0$,
it follows that
$M= C_5\star C_5\star C_5$,
where one of the $5$-cycles is $C_5(v,a,u,w,b)$.

\noindent \textbf{Case 3:} Suppose that $\deg(v)\leq 2d+1$ for all $v\in V(M)$ and that there exists $v\in V(M)$ with $\deg(v)=2d+1$. Then $N_M(v)$ is a discrete $(d-1)$-pseudomanifold on $2d+1$ vertices, and by Theorem~\ref{2d+3 vertex}, we have $N_M(v)= C_5\star S^{d-3}_{2d-4}$. Choose distinct vertices $u_1,u_2,u_3\in V(M)\setminus (V(N_M(v))\cup\{v\})$ such that $V(M)=\{v\}\cup V(N_M(v))\cup\{u_1,u_2,u_3\}$, and by Lemma~\ref{adjacency of vertices}, assume $u_1u_2,u_2u_3\in E(M)$. For $x\in S^{d-3}_{2d-4}$, we have $\deg_{N_M(v)}(x)=2d-1$ and $xv\in E(M)$, so $x$ is adjacent to at most one of $u_1,u_2,u_3$, implying $f_0(N_M(u_1u_2))\leq 6$. However, since $d\geq5$, we get $f_0(N_M(u_1u_2))\geq8$, a contradiction; hence this case cannot occur.

\noindent \textbf{Case 4:} If $\deg(v)=2d$ for all $v\in V(M)$, then by Proposition~\ref{prop-4}, $M= S^{d}_{2d+2}$, contradicting $f_0(M)=2d+5$; hence this case cannot occur.
\end{proof}

\subsection{Discrete $d$-pseudomanifold with $2d+6$ vertices}
Let $M$ be a discrete $d$-pseudomanifold with $2d+6$ vertices. By the definition of a discrete $d$-pseudomanifold, when $d=1$ and
$f_0(M)=8$, we have $M= C_8$. For $d\geq2$, we have the following results.

\begin{figure}[ht]
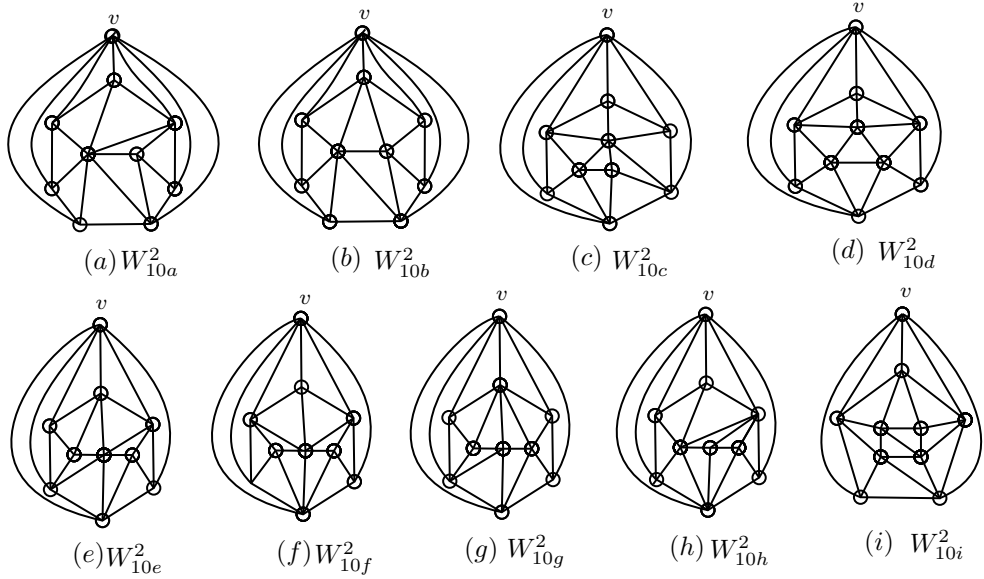

   \begin{center}

\tikzset{every picture/.style={line width=0.75pt}} 


\caption{Some discrete $2$-pseudomanifolds on $10$ vertices.}
\label{fig:2}
   \end{center}
\end{figure}

\begin{theorem}\label{2dim 10 vertex}
Let $M$ be a discrete  $2$-pseudomanifold with $10$ vertices.
Then, up to isomorphism,
\[
M= C_8\star S_2^0,
 W^2_{10a}, W^2_{10b}, W^2_{10c}, W^2_{10d}, W^2_{10e}, W^2_{10f}, W^2_{10g}, W^2_{10h}, \text{ or } W^2_{10i}.
\]
where  $W^2_{10a}, W^2_{10b}, W^2_{10c}, W^2_{10d}, W^2_{10e}, W^2_{10f}, W^2_{10g}, W^2_{10h},$ and $W^2_{10i} $ are the graphs depicted
in Figure~\ref{fig:2}.
\end{theorem}

\begin{proof}
Let $M$ be a discrete $2$-pseudomanifold. By Lemma~\ref{lem-6}, $\deg(v)\in\{4, 5, 6, 7, 8\}$ for all $v\in V(M)$. 

\noindent \textbf{Case 1:} Suppose there exists $v \in V(M)$ with $\deg(v)=8$. Then, by Lemma~\ref{lem-4}, $M = C_8 * S_2^0$.

\noindent \textbf{Case 2:} Suppose that $\deg(x)\leq 7$ for all $x \in V(M)$ and let $\deg(v)=7$. Then $N_M(v)= C_7$. Let $u_1, u_2 \notin V(N_M(v))$ be the remaining vertices. By Lemma~\ref{adjacency of vertices}, we have $u_1u_2 \in E(M)$. It is now straightforward to verify that the only possible structures for $M$ are $W^2_{10a}$ and $W^2_{10b}$.

\noindent  \textbf{Case 3:} 
Suppose that $\deg(x)\leq 6$ for all $x \in V(M)$ and let $\deg(v)=6$. Then $N_M(v)= C_6$. Let $u_1,u_2,u_3 \in V(M)\setminus V(N_M(v))$ be the remaining vertices. By Lemma~\ref{adjacency of vertices}, at least two of $u_1u_2$, $u_2u_3$, and $u_1u_3$ are edges. If $u_1u_2, u_2u_3, u_1u_3 \in E(M)$, then the only possible structures for $M$ are $W^2_{10c}$ and $W^2_{10d}$. If $u_1u_2, u_2u_3 \in E(M)$ and $u_1u_3 \notin E(M)$, then the only possible structures for $M$ are $W^2_{10e}$, $W^2_{10f}$, $W^2_{10g}$, and $W^2_{10h}$.

\noindent  \textbf{Case 4:} 
Suppose that $\deg(x)\leq 5$ for all $x \in V(M)$ and let $\deg(v)=5$. Then $N_M(v)= C_5$. Let $u_1,u_2,u_3,u_4 \in V(M)\setminus V(N_M(v))$ be the remaining vertices. Then the only possible structure for $M$ is $W^2_{10i}$.

\noindent   \textbf{Case 5:} Suppose that $\deg(v) = 4$ for all $v \in V(M)$. Then, by Proposition~\ref{prop-4}, $M= S^2_6$. This contradicts the assumption that $f_0(M) = 10$. Therefore, this case cannot occur.
\end{proof}

For $d\geq 3$, let  $M$ be a discrte $d$-pseudomanifold with $2d+6$ vertices. By Lemma~\ref{lem-6}, the degree of any vertex $v \in V(M)$ satisfies
\[
\deg(v) \in \{2d+4,\, 2d+3,\, 2d+2,\, 2d+1,\, 2d\}.
\]

\begin{lemma}\label{not exist 2d}
For $d \geq 3$, there does not exist a discrete $d$-pseudomanifold $M$ with at least $2d+6$ vertices such that $\deg(v)\leq 2d+1$ for every vertex $v\in V(M)$.
\end{lemma}

\begin{proof}
If $\deg(v)=2d$ for all $v\in V(M)$, then by Proposition~\ref{prop-4}, we obtain $M = S^{d}_{2d+2}$, contradicting $f_0(M)\geq 2d+6$. Hence there exists a vertex $v$ with $\deg(v)=2d+1$, and by Theorem~\ref{2d+3 vertex}, we have $N_M(v) = C_5 * S^{d-3}_{2d-4}$.

Let $u_1,\dots,u_k \in V(M)\setminus (V(N_M(v)) \cup \{v\})$ so that $V(M)=\{v\}\cup V(N_M(v))\cup\{u_1,\dots,u_k \}$, $k\geq 4$. By Lemma~\ref{adjacency of vertices}, there exist at least two edges among the vertices $u_1,\ldots,u_k$; that is, $u_i u_j \in E(M)$ for at least two distinct pairs $(i,j)$. Let $x \in S^{d-3}_{2d-4}$. Then $\deg_{N_M(v)}(x)=2d-1$, and since $xv\in E(M)$, $x$ is not contained in any $N_M(u_i u_j)$.

Therefore, each vertex in $S^{d-3}_{2d-4}$ is adjacent to at most one $u_i$, and each vertex in $C_5$ is adjacent to at most two among $u_1,\dots,u_k$. Hence $u_i u_j u_k$ cannot form a triangle, since $N_{N_M}(u_iu_j u_k)$ cannot contain a vertex from $N_M(v)= C_5 * S^{d-3}_{2d-4}$. Therefore, $N_M(u_i u_j) \subset C_5$ for every possible edge $u_i u_j\in E(M)$. Thus, $d=3$, and $N_M(u_i u_j)=C_5$ for every possible edge $u_i u_j\in E(M)$. Since there are at least two edges among $u_1,\dots,u_k$, we obtain a contradiction on the degree of a vertex of $C_5$. This completes the proof.
\end{proof}

\begin{lemma}\label{2d+6 d(v)=2d+4}
For $d \geq 3$, let $M$ be a discrete $d$-pseudomanifold with $2d+6$ vertices. Suppose that there exists a vertex $v \in V(M)$ such that $\deg(v)=2d+4$. Then, up to isomorphism, 
\[
M = M^{d-1}_{2d+4} \star S^0_{2},
\]
where $M^{d-1}_{2d+4}$ is a discrete $(d-1)$-pseudomanifold with $2d+4$ vertices.
\end{lemma}

\begin{proof}
Since \(\deg(v)=2d+4\), by Lemma~\ref{lem-4}, we obtain $M = M^{d-1}_{2d+4}\star S^0_{2}.$
\end{proof}

\begin{lemma}\label{edge-contraction 1}
For \(d\geq 2\), let \(M\) be a discrete \(d\)-pseudomanifold. Let \(v,u,w\in V(M)\) satisfy
\(V(M)=V(N_M(v))\cup\{v,u,w\}\), where \(u,w\notin V(N_M(v))\).
Then \(uw\in E(M)\), and the contraction of the edge \(uw\) yields a
discrete \(d\)-pseudomanifold.
\end{lemma}

\begin{proof}
It follows from Lemma~\ref{adjacency of vertices}, that \(uw\in E(M)\).
It remains to show that the edge \(uw\) is not contained in an induced
4-cycle. Suppose, for contradiction, that there exists an induced
4-cycle \(C_4(u,x,y, w)\) such that neither \(xw\) nor \(yu\) is an edge in
\(M\). Then \(u,w\notin N_M(xy)\). Therefore,
\(N_M(xy)=N_{N_M(v)}(xy)\star\{v\}\), which contradicts
Corollary~\ref{lem-2}. Hence, the edge \(uw\) is not contained in an induced
4-cycle, and consequently, the contraction of the edge \(uw\) yields a
discrete \(d\)-pseudomanifold.
\end{proof}

\begin{lemma}\label{2d+6 d(v)=2d+3}
For $d \geq 3$, let $M$ be a discrete $d$-pseudomanifold with $2d+6$ vertices. Suppose that $\deg(x) \leq 2d+3$ for all $x \in V(M)$, and that there exists a vertex $v \in V(M)$ such that $\deg(v)=2d+3$. Then, up to isomorphism, 
\begin{align*}
M =& Exp( W_{9a}^2\star S_2^0) \ \text{or}\ Exp( W_{9b}^2 \star S_2^0) \ \text{or}\ Exp(W_{9c}^2 \star S^2_0) \ \text{or}\ C_7 \star C_5 \ \text{for } d=3, \\
M =&  C_5 \star W_{9a}^2 \ \text{or} \ C_5 \star W_{9b}^2 \ \text{or} \ C_6 \star W_8^2\ \text{or}\ Exp(W_{11a}^3\star S^2_0)\ \text{or}\ Exp( W_{11b}^3 \star S^2_0)\ \text{or}\\
&M^2_9\star C_5 \ \text{for}\ d=4, \\
M =& Exp(W_8^2\star C_5 \star S^2_0) \ \text{or}\ M_{11}^3\star C_5 \ \text{for } d=5, \\
M =& C_5 \star C_5 \star W_8^2 \ \text{or}\ M_{13}^4\star C_5 \ \text{for } d=6, \\
M =& M_{15}^5 \star C_5 \ \text{for}\  d = 7.
\end{align*}
\end{lemma}
\begin{proof}
$N_M(v)$ is a discrete $(d-1)$-pseudomanifold on $(2d+3)$ vertices. By Theorem~\ref{2d+5,d>=5}, $N_M(v)= M_{15}^5\star S^{d-7}_{2d-12}$. By Lemma~\ref{adjacency of vertices}, there exist $u,w\in V(M)\setminus V(N_M(v))$ with $uw\in E(M)$ and $\{v,u,w\}\cup V(N_M(v))=V(M)$. Since every $x\in V(S^{d-7}_{2d-12})$ satisfies $\deg(x)=2d+1$ and $xv\in E(M)$, $x$ cannot be adjacent to both $u$ and $w$; hence $N_M(uw)\subseteq M_{15}^5$ hence $d\le7$. If $N_M(v)= N_{2d+1}^{d-2}\star\{a,b\}$, where $N_{2d+1}^{d-2}$ is a discrete $(d-2)$-pseudomanifold on $(2d+1)$ vertices, then neither $a$ nor $b$ is adjacent to both $u$ and $w$, otherwise $\deg(a)=\deg(b)=2d+4$. Thus $N_M(uw)\subseteq N_{2d+1}^{d-2}$. Therefore, by Lemma \ref{prop-5}, $N_M(uw)= N_{2d+1}^{d-2}$. Further, $au,aw\notin E(M)$ would imply $N_M(a)= N_{2d+1}^{d-2}\star v$, contradicting Corollary~\ref{lem-2}; similarly for $b$, so without loss of generality $au,bw\in E(M)$, and hence $M= N_{2d+1}^{d-2}\star C_5(v,a,u,w,b)$ for all $3\leq d\leq 7$. Now we will look the other possibilities of $N_M(v).$ For \(d=3\), Theorem~\ref{2dim 9 vertex} implies that 
\(N_M(v)= W_{9a}^2, W_{9b}^2\), or \(W_{9c}^2\). By Lemma \ref{edge-contraction 1}, contracting the edge \(uw\) 
to a vertex \(p\) gives \(W_{9j}^2\star\{v,p\}\) we obtain $M= Exp(W_{9j}^2\star \{v,p\})$, for \(j\in\{a,b,c\}\), with respect to the vertex \(p\) and an induced cycle \(W\) in 
\(W_{9j}^2\).

For $d=4$, Theorem~\ref{3 dim 11 vertex} yields $N_M(v)= C_6\star C_5$, $W_{11a}^3$, or $W_{11b}^3$. If $N_M(v)=C_6(x_1,\ldots,x_6)\star C_5(y_1,\ldots,y_5)$, then $6\le f_0(N_M(uw))\leq 10$, since $f_0(N_M(uw))=11$ contradicts Corollary~\ref{lem-2}$;$ the only possible induced subgraphs are $C_5\star S_2^0$ and $C_6\star S_2^0$. If $N_M(uw)=C_5\star\{x_1,x_3\}$, $C_5\star\{x_1,x_4\}$, or $C_6\star\{y_1,y_3\}$. Then upto isomorphism $M= v\star(C_6\star C_5)\cup u\star\bigl(C_5\star\{x_1x_2,x_2x_3\}\bigr)\cup w\star\bigl(C_5\star\{x_3x_4,x_4x_5,x_5x_6,x_6x_1\}\bigr) \cup \{uw\},$ which is isomorphic to $C_5 \star W_{9a}^2$ or $M= v\star(C_6\star C_5)\cup u\star\bigl(C_5\star\{x_1x_2,x_2x_3,x_3x_4\}\bigr)\cup w\star\bigl(C_5\star\{x_4x_5,x_5x_6,x_6x_1\}\bigr)\cup \{uw\}$, which is isomorphic to $C_5 \star W_{9b}^2$ or $M= v\star(C_6\star C_5)\cup u\star\bigl(C_6\star\{y_1y_2,y_2y_3\}\bigr)\cup w\star\bigl(C_6\star\{y_3y_4,y_4y_5,y_5y_1\}\bigr)\cup \{uw\}$, which is isomorphic to $C_6\star W_8^2$. By Lemma \ref{edge-contraction 1},these are $ Exp(C_6\star C_5\star \{v,p\})$. If $N_M(v)= W_{11a}^3$ or $W_{11b}^3$, then by Lemma \ref{edge-contraction 1}, contracting $uw$ to $p$ yields  $M= Exp(W_{11j}^3 \star \{v,p\})$ for $j\in\{a,b\}$, with respect to the vertex \(p\) and an induced discrete 2-pseudomanifold \(W\) in 
\(W_{11j}^3\). 

For $d=5$, Theorem~\ref{4dim with 13 vertices} gives $N_M(v)= W_8^2\star C_5$; by Lemma \ref{edge-contraction 1}, contracting $uw$ to $p$ yields $M= Exp(W_8^2\star C_5 \star \{v,p\})$, with respect to the vertex \(p\) and an induced discrete 3-pseudomanifold \(W\) in \(W_8^2\star C_5 \).

For $d=6$, Theorem~\ref{2d+5,d>=5} gives $N_M(v)= C_5\star C_5\star C_5$. By the classification of $4$-dimensional discrete pseudomanifolds, $C_5\star C_5\star S_2^0$ is the unique possible induced subgraph. Writing $N_M(v)=C_5\star C_5\star C_5(a,b,c,d,e)$, we obtain $M=v\star(C_5\star C_5\star C_5)\cup u\star(C_5\star C_5\star\{ab,bc\})\cup w\star(C_5\star C_5\star\{cd,de,ea\})\cup uw\star(C_5\star C_5\star\{ac\})= C_5\star C_5\star W_8^2$. Finally, for $d=7$, Theorem~\ref{2d+5,d>=5} yields $N_M(v)= M_{15}^5\star\{a,b\}$; hence, by the argument that we gave first, $M= M_{15}^5\star C_5$ for $d=7$.
\end{proof}

\begin{lemma}\label{edge-contraction 2}
For \(d\geq 3\), let \(M\) be a discrete \(d\)-pseudomanifold. Let \(v,u_1,u_2,u_3\in V(M)\) satisfy
\(V(M)=V(N_M(v))\cup\{v,u_1,u_2,u_3\}\), where \(u_1,u_2,u_3\notin V(N_M(v))\), \(u_1u_2,u_2u_3\in E(M)\), and \(u_1u_3\notin E(M)\).
Then, the contraction of the edge \(u_1u_2\) yields a
discrete \(d\)-pseudomanifold.
\end{lemma}

\begin{proof}
It remains to show that the edge \(u_1u_2\) is not contained in an
induced 4-cycle. Suppose, for contradiction, that there exists an induced
4-cycle \(C_4(u_1,x,y,u_2)\) such that neither \(xu_2\) nor \(yu_1\) is an
edge of \(M\). Then \(u_1,u_2\notin N_M(xy)\), and hence, by Corollary \ref{lem-2},
\(N_M(xy)=N_{N_M(v)}(xy)\star\{v,u_3\}\). Moreover, \(N_M(x)\) is a
discrete \((d-1)\)-pseudomanifold. Also,
\(V(N_M(x))=\{v\}\cup V(N_{N_M(v)}(x))\cup\{u_1,u_3\}\), where
\(u_1,u_3\notin V(N_{N_M(v)}(x))\).
Therefore, by Lemma~\ref{adjacency of vertices}, we obtain
\(u_1u_3\in E(N_M(x))\), which contradicts the hypothesis. Hence, the
edge \(u_1u_2\) is not contained in an induced 4-cycle, and the
contraction of the edge \(u_1u_2\) yields a discrete
\(d\)-pseudomanifold.
\end{proof}

\begin{lemma}\label{2d+6 d(v)=2d+2}
For $d \geq 3$, let $M$ be a discrete $d$-pseudomanifold with $2d+6$ vertices. Suppose that $\deg(x) \leq 2d+2$ for all $x \in V(M)$, and that there exists a vertex $v \in V(M)$ such that $\deg(v)=2d+2$. Then, $d=3$, and up to isomorphism, 
$M = C_6 \star C_6 \ \text{or}\  W_{12}^3$, where $W_{12}^3$ is the graph depicted in Figure \ref{fig:2}. 
\end{lemma}
\begin{proof}
$N_M(v)$ is a discrete $(d-1)$-pseudomanifold on $2d+2$ vertices. By Theorem~\ref{2d+4,d>=4}, $N_M(v)= M_{10}^3\star S^{d-5}_{2d-8}$. Let $u_1,u_2,u_3\in V(M)\setminus\bigl(V(N_M(v))\cup\{v\}\bigr)$ be such that $\{v\}\cup V(N_M(v))\cup\{u_1,u_2,u_3\}=V(M)$. By Lemma~\ref{adjacency of vertices}, we may assume $u_1u_2,u_2u_3\in E(M)$. Since every $x\in V(S^{d-5}_{2d-8})$ satisfies $\deg_{N_M(v)}(x)=2d$ and $xv\in E(M)$, $x$ is adjacent to at most one of $u_1,u_2,u_3$. Therefore, $N_M(u_1u_2)\subseteq M_{10}^3 \cup \{u_3\}$. Hence, $d\le 5$.

For $d=3$, Theorem~\ref{dim-2,vertex-8} gives $N_M(v)= C_6\star\{a,b\}$ or $W_8^2$. Assume $u_1u_3\notin E(M)$. If $N_M(v)= C_6\star\{a,b\}$, then $a$ and $b$ can be adjacent to at most one $u_i$ ($i\in\{1,2,3\}$), otherwise $\deg(a),\deg(b)\ge 9$, a contradiction. Hence, by Lemma~\ref{prop-5}, $N_M(u_1u_2)=N_M(u_2u_3)=C_6$. If $a$ (resp.\ $b$) is not adjacent to any of $u_1,u_2,u_3$, then $N_M(a)= v\star C_6$ (resp.\ $N_M(b)= v\star C_6$), contradicting Corollary~\ref{lem-2}. Thus, both $a$ and $b$ are adjacent to at least one $u_i$, but neither can be adjacent to $u_2$, otherwise $\deg(u_2)\ge 9$. Without loss of generality, $au_1,bu_3\in E(M)$, and hence $M= C_6\star C_6(v,a,u_1,u_2,u_3,b)$.

If $N_M(v)= W_8^2$, then by lemma \ref{edge-contraction 2}, contracting the edge $u_1u_2$ to a single vertex $p$ gives a graph isomorphic to either $W_{11a}^3$ or $W_{11b}^3$. Let the resulting graph be $M'$. After this contraction, the degree of each vertex in the cycle $N_M(u_1u_2)$ decreases by $1$ in $M'$. However, in $W_{11b}^3$, there is no induced $n$-cycle, $n\geq 4$, such that every vertex of the cycle has degree less than $8$. Hence, $M$ cannot be isomorphic to $Exp(W_{11b}^3)$. If $M= Exp(W_{11a}^3)$, with respect to the vertex \(p\) and an induced cycle \(W\) in \(N_{W_{11a}^3}(p)\) such that the degree of each vertex of \(W\) in 
\(W_{11a}^3\) is less than \(8\), then \(p\) cannot be the vertex \(u\) of \(W_{11a}^3\), as shown in Figure~\ref{fig:3}. Moreover, since \(u_3\) cannot 
belong to \(W\), the vertex \(p\) cannot be the vertex \(w\) of \(W_{11a}^3\), as shown in Figure~\ref{fig:3}. Therefore, this case is not possible.

Now assume that \(u_1u_3\in E(M)\). Then the case
\(N_M(v)= C_6\star\{a,b\}\) cannot occur. Indeed, since
\(au_1\in E(M)\), we would have
\(N_M(a)= C_6\star\{u_1,v\}\), which implies that
\(\deg(u_1)\geq 9\), a contradiction.
Suppose that \(N_M(v)= W_8^2\). Then
\(N_M(u_1u_2u_3)=\{x,y\}\), where \(x\) and \(y\) are the degree \(4\)
vertices of \(W_8^2\). There are exactly four distinct paths
\(P_1,P_2,P_3,P_4\) between \(x\) and \(y\) in \(W_8^2\), where two of
these paths contain three vertices and the other two contain four
vertices. Note that \(N_M(u_1u_2)=P_i\cup\{xu_3,yu_3\}\) for some \(i\).
The edges \(u_1u_2\), \(u_2u_3\), and \(u_1u_3\) correspond to the paths
\(P_3,P_1,P_2\), respectively. Therefore, the vertices of the remaining
path \(P_4\), other than \(x\) and \(y\), are adjacent to exactly one of
\(u_1,u_2,u_3\). If \(P_4\) contains two vertices \(p,q\), then
\(N_M(p)=N_{N_M(v)}(p)\star\{v,u_i\}\) and
\(N_M(q)=N_{N_M(v)}(q)\star\{v,u_i\}\). Hence,
\(\deg(u_i)\geq 9\), which is a contradiction. Therefore, \(P_4\) contains
three vertices, and two of \(P_1,P_2,P_3\) contain four vertices. Thus,
we obtain \(M= W_{12}^3\), as depicted in Figure~\ref{fig:3}.

Let $d=4$. By Theorem~\ref{3dim with 10 vertices}, $N_M(v)= W_8^2\star S_2^0$, $C_6\star\{a,b\}\star\{c,d\}$, or $C_5\star C_5$. Every $x\in V(S_2^0)$ and each of $a,b,c,d$ is adjacent to at most one $u_i$ ($i\in\{1,2,3\}$), otherwise its degree exceeds $11$, and to at least one $u_i$, else Corollary~\ref{lem-2} is contradicted. If $u_1u_3\in E(M)$, then in the first case $N_M(x)= W_8^2\star\{v,u_i\}$, giving $\deg(u_i)\ge11$; in the second, we will have $N_M(u_1u_2u_3)= C_6$ by Lemma \ref{prop-5}, hence $N_M(u_iu_j)= C_6 \star \{u_k\}$ for all $i,j,k \in \{1,2,3\}$, $i\neq j\neq k$, a contradiction of Corrolary \ref{lem-2}   and in the third, some $x\in V(C_5\star C_5)$ is adjacent to three $u_i$'s, implying $\deg(x)\ge11$, a contradiction. Hence $u_1u_3\notin E(M)$, so only $u_1u_2,u_2u_3\in E(M)$. If $N_M(v)= W_8^2\star\{a,b\}$, then $a,b$ are adjacent to at most one $u_i$ as mentioned above, so $N_M(u_1u_2),N_M(u_2u_3)\subseteq W_8^2$. Therefore, the vertex $x\in V(W_8^2)$, which has degree $5$ in $W_8^2$, must have degree $11$ in $M$, contradicting the assumption. Hence, this case cannot occur. If $N_M(v)= C_6\star S_4^1$, then every $x\in V(S_4^1)$ is adjacent to at most one $u_i$ as mentioned above, implying $N_M(u_1u_2)\subseteq C_6$, contrary to Lemma~\ref{prop-5}. Finally, if $N_M(v)= C_5\star C_5$, then Theorem~\ref{vertex edge lowerbound} gives $f_0(N_M(u_iu_j))\ge6$, forcing some $x\in V(C_5)$ to satisfy $xu_1,xu_2,xu_3\in E(M)$, and hence $\deg(x)\ge11$, a contradiction.

Let $d=5$ or $6$. By Theorem~\ref{2d+4,d>=4}, $N_M(v)= W_8^2\star S_{2d-6}^{d-4}$, $C_6\star S_{2d-4}^{d-3}$, or $C_5\star C_5\star\ S^{d-5}_{2d-8}$. In each case, every vertex outside $W_8^2$, $C_6$, or $C_5\star C_5$ is adjacent to at most one of $u_1,u_2,u_3$, else the degree bound is violated. Hence, if $u_1u_3\notin E(M)$, then $N_M(u_1u_2)\subseteq W_8^2$ or $C_6$, contradicting Lemma \ref{prop-5} in both cases $d=5,6$ and  $N_M(u_1u_2),N_M(u_2u_3)\subseteq C_5\star C_5$, yields a contradiction for $d=6$ by Lemma \ref{prop-5}, for $d=5$ by lemma \ref{prop-5} $N_M(u_1u_2)=N_M(u_2u_3)= C_5\star C_5$ yielding a vertex of degree $13$, a contradiction. If $u_1u_3\in E(M)$, then $N_M(u_1u_2u_3)\subseteq W_8^2$, $C_6$, or $C_5\star C_5$; for $d=6$ all cases contradict Lemma \ref{prop-5}. Now by Lemma~\ref{prop-5}, equality holds in the first case for $d=5$, yielding a vertex of degree $13.$  The second case contradicts Lemma~\ref{prop-5}, while third case also yielding a vertex of degree $13$. Thus, $d=5$ and $d=6$ are impossible.
\end{proof}

\begin{theorem}\label{2d+6,d>=3}
Let \(M\) be a discrete \(d\)-pseudomanifold with \(2d+6\) vertices, where
\(d\geq 3\). Then, up to isomorphism, \(M\) is one of the discrete
\(d\)-pseudomanifolds obtained in Lemmas \ref{2d+6 d(v)=2d+4},
\ref{2d+6 d(v)=2d+3}, and \ref{2d+6 d(v)=2d+2}.
\end{theorem}

\section{Equivalence of Discrete \(d\)-Pseudomanifolds and Flag Normal \(d\)-Pseudomanifolds}
A \emph{simplicial complex} $\Delta$ on a vertex set $V=V(\Delta)$ is a collection of subsets 
$\sigma\subseteq V$, called faces, that is closed under inclusion. For $\sigma\in\Delta$, 
we define the \emph{dimension} of $\sigma$ by \emph{$\dim \sigma:=|\sigma|-1$} and define the 
dimension of $\Delta$, denoted by \emph{$\dim \Delta$}, as the maximum dimension of its 
faces. A \emph{facet} of $\Delta$ is a maximal face with respect to inclusion, and $\Delta$ is 
said to be pure if all of its facets have the same dimension. If $\Delta$ is a 
simplicial complex and $\sigma$ is a face of $\Delta$, then the \emph{link} of $\sigma$ in 
$\Delta$ is defined as 
\emph{$\operatorname{lk}_{\Delta}(\sigma):=\{\tau\setminus\sigma\in\Delta:
\sigma\subseteq\tau\in\Delta\}$}.
If $\Delta$ and $\Delta'$ are two simplicial complexes on disjoint vertex sets, then the 
join of $\Delta$ and $\Delta'$, denoted by $\Delta \star \Delta'$, is the simplicial complex 
on the vertex set $V(\Delta)\sqcup V(\Delta')$ whose faces are given by $\{\sigma\cup\tau:\sigma\in\Delta,\ \tau\in\Delta'\}.$ For a simplicial complex $\Delta$, let $f_0(\Delta)$ and $f_1(\Delta)$ denote the number of vertices and edges of $\Delta$, respectively.

A $d$-dimensional simplicial complex $\Delta$ is called a \emph{simplicial $d$-manifold} (respectively, a \emph{simplicial $d$-sphere}) if its geometric realization $\|\Delta\|$ is homeomorphic to a topological manifold (respectively, to a sphere). In this case, $\Delta$ is said to be a triangulation of $\|\Delta\|$.
A simplicial complex $\Delta$ is called 
\emph{flag} if all minimal non-faces of $\Delta$, also called missing faces, have 
cardinality two. Equivalently, $\Delta$ is the clique complex of its edge graph.
\begin{definition}
{\rm
    A \(d\)-dimensional simplicial complex \(\Delta\) is called a \emph{weak \(d\)-pseudomanifold} if it is pure and every \((d-1)\)-dimensional face (called a \emph{ridge}) of \(\Delta\) is contained in exactly two facets.  

A weak \(d\)-pseudomanifold \(\Delta\) is called a \emph{normal \(d\)-pseudomanifold} if it is \emph{connected}, and the link of each face of dimension at most \(d-2\) is  connected.}
\end{definition}

A vertex of a normal \(d\)-pseudomanifold \(\Delta\) is called \emph{non-singular} if its link is a simplicial \((d-1)\)-sphere, and \emph{singular} otherwise.
From the definition of a normal $d$-pseudomanifold, the following result follows, which is well known in the literature.

\begin{proposition} \label{prop-1}
Let $\Delta$ be a connected pure $d$-dimensional simplicial complex. Then $\Delta$ is a 
\emph{normal $d$-dimensional pseudomanifold} if and only if, for every vertex 
$v\in V(\Delta)$, the link $\operatorname{lk}_{\Delta}(v)$ is a 
$(d-1)$-dimensional normal pseudomanifold.
\end{proposition}

Since a flag normal $d$-pseudomanifold is the clique complex of its edge graph, the following result follows inductively.

\begin{proposition} \label{prop-2}
A graph $M$ is a discrete $d$-pseudomanifold if and only if it is the edge graph of a flag normal $d$-pseudomanifold.
\end{proposition}

If a discrete $d$-pseudomanifold $M$ arises as the edge graph of a flag normal $d$-pseudomanifold $\Delta$, then we say that $M$ represents $\Delta$, and that $\Delta$ corresponds to $M$.
The pseudomanifold version of Theorem~\ref{vertex edge lowerbound} was
proved in \cite[Proposition~2.2]{Athanasiadis}. Similarly, the
pseudomanifold versions of Theorems~\ref{prop-4} and
\ref{2d+3 vertex} were established in \cite[Lemma~3.2]{Zheng2017}.
Moreover, the flag \(2\)-spheres with at most \(9\) vertices obtained from
Theorems~\ref{dim-2,vertex-8} and \ref{2dim 9 vertex} can be verified
using the classification in \cite{Walkup}. Since we have fully classified discrete $d$-pseudomanifolds with at most $2d+6$ vertices, 
we obtain a complete classification of flag normal $d$-pseudomanifolds with at most 
$2d+6$ vertices. Indeed, a flag normal $d$-pseudomanifold is precisely the clique 
complex of its underlying discrete $d$-pseudomanifold. Moreover, every such flag normal 
$d$-pseudomanifold is either a join of simplicial $d$-spheres or an edge expansion of a simplicial $d$-sphere; 
hence, it is a simplicial $d$-sphere. Therefore, we have the following result.

\begin{theorem}\label{thm:flagclassificiation}
Let $\Delta$ be a flag normal $d$-pseudomanifold with at most $2d+6$ vertices. Then 
$\Delta$ is a simplicial $d$-sphere. Moreover, a complete combinatorial 
description of all such spheres is obtained from Corollary \ref{2d+2 vertex} and Theorems \ref{2d+3 vertex}, \ref{dim-2,vertex-8}, \ref{3dim with 10 vertices}, \ref{2d+4,d>=4}, \ref{2dim 9 vertex}, \ref{3 dim 11 vertex}, \ref{4dim with 13 vertices}, \ref{2d+5,d>=5}, \ref{2dim 10 vertex}, and \ref{2d+6,d>=3}.
\end{theorem}

\begin{remark}{\rm
There exists a flag triangulation of $\mathbb{RP}^{2}$ with $11$ vertices (cf. \cite{Zheng2020}), whose facet 
list is
\[
\begin{aligned}
&ade,\ aef,\ abf,\ bfg,\ bcg,\ cdg,\ dgh,\ adh,\ aik,\ efk,\\
&fgk,\ ghk,\ hik,\ ijk,\ ejk,\ 
abi,\ bij,\ bcj,\ cdj,\ dej.
\end{aligned}
\]
We denote this simplicial $2$-manifold by $\mathbb{RP}^{2}_{11}$. Moreover, this triangulation is minimal with respect to the number of vertices. Hence, 
there exists a simplicial $d$-manifold (respectively, a normal $d$-pseudomanifold) with 
$2d+7$ vertices that is not a simplicial $d$-sphere for $d=2$.

Furthermore, by taking the suspension of this simplicial complex, that is, by taking its join with
\[
\{a_1,b_1\}\star\{a_2,b_2\}\star\cdots\star\{a_{d-2},b_{d-2}\}=S^{d-3}_{2d-4},
\]
we obtain a normal $d$-pseudomanifold that is not a simplicial $d$-sphere. Therefore, the sphere characterization is optimal within the class of flag normal 
$d$-pseudomanifolds, since there exist examples on $2d+7$ vertices that fail to be 
simplicial $d$-spheres.
}    
\end{remark}

\begin{lemma}\label{not exist 2d+2}
For $d \geq 3$, there does not exist a discrete $d$-pseudomanifold $M$ with $2d+7$ vertices such that $\deg(v)\leq 2d+2$ for every vertex $v\in V(M)$.
\end{lemma}
\begin{proof}
By Lemma~ \ref{not exist 2d}, it is impossible that $deg(x)\leq 2d+1$ for every $x\in V(M)$. Hence there exists $v\in V(M)$ with $deg(v)=2d+2$. Let $u_1,\ldots,u_4\in V(M)\setminus \bigl(V(N_M(v))\cup\{v\}\bigr)$ so that $V(M)=\{v\}\cup V(N_M(v))\cup\{u_1,\ldots,u_4\}$. For $d\geq 4$, we have $N_M(v)= M_{10}^3\star S_{2d-8}^{d-5}$, and by Theorem~\ref{3dim with 10 vertices}, $N_M(v)= W_8^2\star S_2^0\star S_{2d-8}^{d-5}$, $C_6\star C_4\star S_{2d-8}^{d-5}$, or $C_5\star C_5\star S_{2d-8}^{d-5}$. Any vertex of $N_M(v)$ outside $W_8^2$, $C_6$, or $C_5\star C_5$ can be adjacent to at most one of $u_1,\ldots,u_4$, since otherwise its degree would be at least $2d+3$, a contradiction. Moreover, any vertex of $W_8^2$, $C_6$, or $C_5\star C_5$ can be adjacent to at most three of $u_1,\ldots,u_4$, by the degree bound $deg(x)\leq 2d+2$. For $d=3$, Theorem~\ref{dim-2,vertex-8} gives $N_M(v)= C_6\star S_2^0$ or $W_8^2$. In the former case, each $x\in V(S_2^0)$ is adjacent to at most one of $u_1,\ldots,u_4$, otherwise $deg(x)\geq 9$, a contradiction; similarly, every vertex of $W_8^2$ or $C_6$ is adjacent to at most three of $u_1,\ldots,u_4$. By Lemma \ref{lem-2}, for each fixed $i$, there will exist some $j$ such that $u_iu_j \in E(M)$, $i,j \in\{1,2,3,4\}.$

 \noindent   \textbf{Case 1:} Suppose \(\{u_i,u_j,u_k\}\) does not form a triangle for any distinct \(i,j,k \in \{1,2,3,4\}\). Then, for \(d\geq 4\), it follows from the above discussion that, for every edge \(u_i u_j\), the graph \(N_M(u_i u_j)\subseteq W_8^2\), \(C_6\), or \(C_5\star C_5\), according as \(N_M(v)\) is of the first, second, or third type. Now, for \(d\geq 4\), \(N_M(u_i u_j)\subseteq C_6\) is not possible, as it contradicts Lemma \ref{prop-5}. Also, \(N_M(u_i u_j)\subseteq W_8^2\) implies \(d=4\), and then \(N_M(u_i u_j)=W_8^2\) by Lemma \ref{prop-5}. Hence, there exists a vertex in \(W_8^2\) which is adjacent to all four vertices \(u_1,\ldots,u_4\), contradicting the fact that any vertex of \(N_M(v)\) can be adjacent to at most three of them. Let \(N_M(u_i u_j)\subset C_5 \star C_5\). Then \(d=4\) or \(d=5\). If \(d=5\), then \(N_M(u_i u_j)=C_5 \star C_5\) for all \(u_i u_j \in E(M)\). Hence, there exists a vertex in \(C_5 \star C_5\) which is adjacent to all four vertices \(u_1,\ldots,u_4\), contradicting the fact that any vertex of \(N_M(v)\) can be adjacent to at most three of them. Now, in this case, if \(d=4\), then \(N_M(u_i u_j)\subseteq C_5 \star C_5\) and \(f_0(N_M(u_i u_j))\geq 6\). Hence, there exists at least one \(x\in V(C_5 \star C_5)\) such that \(x\) is adjacent to three vertices among \(u_1,\dots,u_4\). Then \(\deg(x)\geq 11\), which is a contradiction.
    
Let \(d=3\). If \(N_M(v)= W_8^2\), then \(N_M(u_i u_j)\subseteq W_8^2\) for every \(u_i u_j\in E(M)\). Then \(N_M(u_i u_j)\) is an induced \(4\)-cycle or \(5\)-cycle in \(W_8^2\). Hence, either there exists a vertex \(x\in W_8^2\) that has degree \(4\) in \(W_8^2\) and is adjacent to all four \(u_i\), or there exists a vertex \(x\in W_8^2\) that has degree \(5\) in \(W_8^2\) and is adjacent to at least three \(u_i\), contradicting \(\deg(x)\leq 8\).

Let \(d=3\) and \(N_M(v)=C_6\star S_2^0\). Then \(N_M(u_i u_j)=C_6\) for all possible edges \(u_i u_j\). Hence, for every vertex \(x\in C_6\), \(\deg(x)\geq 9\), which is a contradiction.
    
Therefore, there exist distinct \(i,j,k\in\{1,2,3,4\}\) such that \(u_i u_j,u_j u_k,u_k u_i\in E(M)\). Without loss of generality, assume that \(u_1u_2,u_2u_3,u_3u_1\in E(M)\). By Lemma~\ref{adjacency of vertices}, we may further assume that \(u_2u_4\in E(M)\).

 \noindent  \textbf{Case 2:} Suppose that \(u_1u_4,u_3u_4\in E(M)\).  If $d\geq 4$, then \(N_M(u_1u_2u_3u_4)\subseteq N_M(v)\), contradicting the fact that no vertex of \(N_M(v)\) can be adjacent to all four vertices \(u_1,u_2,u_3,u_4\). Thus, this case is impossible. 
 
Let \(d=3\) and \(N_M(v)=C_6\star\{a,b\}\). Let \(d=3\) and \(N_M(v)=C_6\star\{a,b\}\). Then \(\deg(a)\leq 8\) and 
\(v\star C_6\subseteq N_M(a)\) together imply that $N_M(a)=C_6\star\{v,u_i\}$
for some \(i\). Consequently, \(\deg(u_i)\geq 9\), which is a contradiction.

Let \(d=3\) and \(N_M(v)=W_8^2\). Then \(N_M(u_iu_ju_k)\) contains exactly one vertex of \(N_M(v)\) having degree \(4\) in \(N_M(v)\). Moreover, for distinct triples \(i,j,k\), the corresponding degree \(4\) vertices in \(N_M(v)\) are distinct. Furthermore, \(N_M(u_iu_j)\) contains at least one vertex of degree \(5\) in \(N_M(v)\) in order to form a cycle. Since there are six edges and four vertices of degree \(5\) in \(N_M(v)\), at least one degree \(5\) vertex of \(N_M(v)\) must be adjacent to three of \(u_1,\ldots,u_4\). Hence, this vertex has degree \(9\) in \(M\), which is a contradiction.

 \noindent  \textbf{Case 3:} Suppose \(u_1u_4\in E(M)\) and \(u_3u_4\notin E(M)\). For \(d\geq 4\), our previous analysis implies that \(N_M(u_1u_2u_3)\) and \(N_M(u_1u_2u_4)\) are contained in \(W_8^2\), \(C_6\), or \(C_5\star C_5\). Let \(N_M(u_1u_2u_3), N_M(u_1u_2u_4)\subseteq W_8^2\). Then, there exists \(x\in W_8^2\) with degree 5 in $W_8^2$ such that \(\deg(x)\geq 2d+3\), which is a contradiction. Let \(N_M(u_1u_2u_3), N_M(u_1u_2u_4)\subseteq C_6\). Then \(d=4\), and by Lemma \ref{prop-5}, \(N_M(u_1u_2u_3)=N_M(u_1u_2u_4)=C_6\), a contradiction because then \(\deg(x)\geq 11\) for all \(x\in V(C_6)\). Now suppose \(N_M(u_1u_2u_3), N_M(u_1u_2u_4)\subseteq C_5\star C_5\). But this is not possible, since for every \(x\in V(C_5\star C_5)\), \(x\) can be adjacent to at most two vertices among \(u_1,\dots,u_4\).

Let \(d=3\) and \(N_M(u_1u_2u_3), N_M(u_1u_2u_4) \subseteq C_6\). Let \(N_M(u_1u_2u_3)=\{v_1,v_2\}\) and \(N_M(u_1u_2u_4)=\{v_3,v_4\}\), where $v_1,\dots,v_4$ are four distinct vertices in $C_6$. Therefore, \(N_M(u_1u_2)\) is a \(6\)-cycle $C'_6$ containing \(v_1,\dots,v_4,u_3,u_4\).

If \(\deg(u_1)=7\) then  \(N_M(u_1)=C'_6\star \{u_2\}\), a contradiction. Since \(\deg(u_1)\leq 8\), \(N_M(u_1)\) contains exactly one vertex, say \(y\), of \(C_6\) other than \(v_1,\dots,v_4\). Therefore, \(N_M(u_1)=C'_6\star \{u_2, y\}\). Then \(y\) is a part of \(N_M(u_1u_3)\) and \(N_M(u_1u_4)\). Since \(y\) is not adjacent to \(u_2\), it follows that \(N_M(u_2u_3), N_M(u_2u_4)\) contain three more vertices from \(C_6\). Hence, \(\deg(u_2)\geq 9\), a contradiction.

Let \(d=3\) and \(N_M(u_1u_2u_3), N_M(u_1u_2u_4) \subseteq W_8^2\). Let \(N_M(u_1u_2u_3)=\{v_1,v_2\}\) and \(N_M(u_1u_2u_4)=\{v_3,v_4\}\). Since \(\deg(x)\leq 8\) for every vertex \(x\in V(M)\), the vertices \(v_1,\dots,v_4\) are distinct degree \(4\) vertices in \(W_8^2\). Therefore, \(N_M(u_1u_2)\) is a \(6\)-cycle $C'_6$ containing \(v_1,\dots,v_4,u_3,u_4\). 
If \(\deg(u_1)=7\) then  \(N_M(u_1)=C'_6\star \{u_2\}\), a contradiction. Since \(\deg(u_1)\leq 8\), \(N_M(u_1)\) contains exactly one vertex, say \(y\), of  \(W_8^2\) other than \(v_1,\dots,v_4\). Therefore, \(N_M(u_1)=C'_6\star \{u_2, y\}\). Then \(y\) has degree \(5\) in \(W_8^2\), and it is a part of \(N_M(u_1u_3)\) and \(N_M(u_1u_4)\). Hence, \(\deg(y)\geq 9\), a contradiction. 

 \noindent \textbf{Case 4:} Suppose \(u_1u_4,u_3u_4\notin E(M)\). For \(d\geq 4\), \(N_M(u_1u_2u_3)\) and \(N_M(u_2u_4)\) are contained in \(W_8^2\), \(C_6\), or \(C_5\star C_5\), yielding the same contradiction. For \(d=3\), $N_M(u_1u_2u_3), N_M(u_2u_4)$ are contained in \(W_8^2\) or \(C_6\).

Let \(d=3\) and \(N_M(u_1u_2u_3), N_M(u_2u_4) \subseteq C_6\). Then \(N_M(u_2u_4)=C_6\). Let \(N_M(u_1u_2u_3)=\{v_1,v_2\}\), where \(v_1,v_2\in C_6\). Then \(\deg(v_1)\geq 9\), a contradiction.

Let \(d=3\) and \(N_M(u_1u_2u_3), N_M(u_2u_4) \subseteq W_8^2\). Let \(N_M(u_1u_2u_3)=\{v_1,v_2\}\), where \(v_1,v_2\) are distinct degree \(4\) vertices in \(W_8^2\). Since \(v_1,v_2\) are not adjacent to \(u_4\), \(N_M(u_2u_4)\) is a cycle in \(W_8^2\) containing at least four vertices different from \(v_1,v_2\). Therefore, \(\deg(u_2)\geq 9\), a contradiction.
\end{proof}

\begin{theorem}\label{thm:flag-2d+7}
For $d\geq 2$, let $\Delta$ be a flag normal $d$-pseudomanifold with at most $2d+7$ vertices. Then either $\Delta$ is a simplicial $d$-sphere or a flag triangulation of the $(d-2)$-fold suspension of $\mathbb{RP}^2$.
\end{theorem}

\begin{proof}
If $\Delta$ has $2d+6$ vertices, then, by Theorem~\ref{thm:flagclassificiation}, $\Delta$ is a simplicial $d$-sphere. Assume now that $\Delta$ has $2d+7$ vertices.

We prove the result by induction on the dimension $d$. For $d=2$, it is known that every flag normal $2$-pseudomanifold with at most $11$ vertices is either a simplicial $2$-sphere or a flag triangulation of $\mathbb{RP}^2$ (cf. \cite{Zheng2020}). Therefore, assume that $d\geq 3$ and that the result holds for all dimensions less than $d$.

Let $M$ be the discrete $d$-pseudomanifold corresponding to $\Delta$, that is, the edge graph of $\Delta$. By Lemma~\ref{lem-6}, the degree of each vertex $v\in V(M)$ satisfies $\deg(v)\in \{2d+5,\,2d+4,\,2d+3,\,2d+2,\,2d+1,\,2d\}.$
Furthermore, by Lemma \ref{not exist 2d+2}, the case $\deg(v)\leq 2d+2$ cannot occur. Hence,
\[
\deg(v)\in \{2d+5,\,2d+4,\,2d+3\},
\]
for every vertex $v\in V(M)$.

\noindent\textbf{Case 1.} Suppose there exists a vertex $v\in V(M)$ such that $\deg(v)=2d+5$. Then, by Lemma~\ref{lem-4},
\[
M = M^{d-1}_{2d+5}\star S^0_2.
\]
Consequently, $\Delta$ is the suspension of a flag normal $(d-1)$-pseudomanifold on $2d+5$ vertices. Therefore, the desired conclusion follows from the induction hypothesis.

\noindent \textbf{Case 2:} Suppose that $\deg(v)\leq 2d+4$ for all $v\in V(M)$, and that there exists a vertex $v\in V(M)$ with $\deg(v)=2d+4$. Then there exist vertices $u,w\in V(M)\setminus V(N_M(v))$ such that
$V(M)=\{v\}\cup V(N_M(v))\cup\{u,w\}.$
By Lemma~\ref{adjacency of vertices}, we have $uw\in E(M)$. Therefore, by Lemma~\ref{edge-contraction 1}, the edge $uw$ can be contracted, yielding a discrete $d$-pseudomanifold $M'$ with $2d+6$ vertices. Let $\Delta'$ be the flag normal $d$-pseudomanifold corresponding to $M'$. By Theorem~\ref{thm:flagclassificiation}, $\Delta'$ is a simplicial $d$-sphere. Since $\Delta'$ is obtained from $\Delta$ by contracting an edge, $\Delta$ is itself a simplicial $d$-sphere.

\noindent \textbf{Case 3:} 
Suppose that $\deg(v)\leq 2d+3$ for all $v\in V(M)$, and that there exists a vertex $v\in V(M)$ with $\deg(v)=2d+3$. Then there exist vertices
$u_1,u_2,u_3\in V(M)\setminus V(N_M(v))$ such that
$V(M)=\{v\}\cup V(N_M(v))\cup\{u_1,u_2,u_3\}.$
By Lemma~\ref{adjacency of vertices}, assume that $u_1u_2,\; u_1u_3\in E(M).$

Suppose that \(u_1u_2\) is contained in an induced \(4\)-cycle
\(C_4(u_1,x,y,u_2)\). If \(y=u_3\), then \(u_1u_3\notin E(M)\). Moreover,
$V(N_M(x))= \{v\} \cup V(N_{N_M(x)}(v))\cup\{u_1,u_3\}.$
Therefore, by Lemma~\ref{adjacency of vertices}, we have
\(u_1u_3\in E(N_M(x))\), which is a contradiction. Since $u_1$ and $u_2$ are not adjacent to $v$, we have
$x,y\in V(N_M(v))$, and hence $xy\in E(N_M(v))$. Subdivide the edge $xy$.
Repeating this procedure whenever $u_1u_2$ is contained in an induced $4$-cycle,
we eventually obtain a discrete $d$-pseudomanifold in which $u_1u_2$ is not contained in any induced $4$-cycle. We may then contract the edge $u_1u_2$; let $u$ denote the resulting vertex and let $M'$ be the resulting discrete $d$-pseudomanifold.

Then $u,u_3\notin V(N_{M'}(v))$, $uu_3\in E(M')$, and
$V(M')=\{v\}\cup V(N_{M'}(v))\cup\{u,u_3\}.$
Therefore, by Lemma~\ref{edge-contraction 1}, the edge $uu_3$ can be contracted, yielding a discrete $d$-pseudomanifold $M''$. Moreover, $M''$ is the suspension of $N_{M'}(v)$.

Since $N_{M'}(v)$ is obtained from $N_M(v)$ by a sequence of edge subdivisions, and $N_M(v)$ represents a simplicial $(d-1)$-sphere, it follows that $N_{M'}(v)$ also represents a simplicial $(d-1)$-sphere. Consequently, $M''$ represents a simplicial $d$-sphere. As $M$ is obtained from $M''$ by a sequence of edge expansions and edge contractions, $M$ represents simplicial $d$-spheres. Therefore, $\Delta$ is a simplicial $d$-sphere.
%
%
\end{proof}

Note that, in Case 3 of Theorem~\ref{thm:flag-2d+7}, we can further apply edge contractions to $M''$ along all the new edges in $N_{M'}(v)$ until we obtain $N_M(v)$. This process yields a new discrete $d$-pseudomanifold $M'''$, which is the suspension of $N_{M}(v)$. Moreover, $M'''$ has $2d+5$ vertices. Therefore, we obtain the following result.

\begin{corollary}
For $d\geq 3$, let $M$ be a discrete $d$-pseudomanifold with $2d+7$ vertices. Then either
\[
M = M^{d-1}_{2d+5}\star S^0_2,
\]
or $M$ is obtained from $M^{d-1}_{n}$ by a sequence of edge contractions and edge expansions, where $n\leq 2d+6$.
\end{corollary}

\section{Future Directions}

The present work contributes to the graph-theoretic study of manifold-like structures with singularities through the notion of discrete $d$-pseudomanifolds. In addition to providing a complete combinatorial classification of discrete $d$-pseudomanifolds with at most $2d+6$ vertices, it develops new graph-theoretic tools for their investigation. We also establish a one-to-one correspondence between discrete $d$-pseudomanifolds and the one-skeleta of flag normal pseudomanifolds. Since a flag simplicial complex is uniquely determined by its edge graph, the results obtained here have direct consequences for the theory of flag triangulations. 


We have proved that if $\Delta$ is a flag simplicial $d$-manifold with $2d+7$ vertices, then $\Delta$ is a simplicial $d$-sphere. Therefore, the natural question is to determine how many vertices are needed before a flag triangulated manifold can have a topology different from that of a sphere.

\begin{question}
For each $d\geq 3$, determine the smallest number $\mu_d$ for which there exists a flag triangulated $d$-manifold on $\mu_d$ vertices that is not a simplicial $d$-sphere. Furthermore, determine the possible topological spaces that can arise in this way.
\end{question}

Several results have been obtained on flag homology manifolds in relation to their $\gamma$-vectors. In particular, the first two $\gamma$-numbers are given as follows. Let $\Delta$ be a normal $d$-pseudomanifold. Then
$$
\gamma_1(\Delta)=f_0(\Delta)-2(d+1) \quad \text{and} \quad 
\gamma_2(\Delta)=f_1(\Delta)-(2d-1)f_0(\Delta)+(d-1)(2d+2),
$$
where $f_0(\Delta)$ and $f_1(\Delta)$ denote the number of vertices and edges of $\Delta$, respectively.

\begin{conj}\cite{Gal}\label{Gal}
Let $\Delta$ be a flag simplicial $d$-sphere. Then $\gamma_2(\Delta)\geq 0$.
\end{conj}

This conjecture is known to hold in dimension three (see \cite{Davis-Okun}) and also for several other classes of simplicial spheres (see \cite{Karu, NevoPetersen}). Since we have established that a flag normal $d$-pseudomanifold with at most $2d+7$ vertices is either a flag simplicial sphere obtained via suspensions, edge expansions, and joins, or a triangulation of the $(d-2)$-fold suspension of $\mathbb{RP}^2_{11}$, it follows that $\gamma_2(\Delta) \ge 0$ for every normal $d$-pseudomanifold $\Delta$ with at most $2d+7$ vertices. 

Using the same definitions, one can extend the notions of $\gamma_1$ and $\gamma_2$ to discrete $d$-pseudomanifolds. From Theorem \ref{vertex edge lowerbound}, we already know that $\gamma_1(M)\geq 0$ for any discrete $d$-pseudomanifold $M$. From the classifications obtained up to $2d+7$ vertices, we observe that $\gamma_2(M)\geq0$
for every discrete $d$-pseudomanifold $M$ with $f_0(M)\le 2d+7$. This leads us to the following conjecture.

\begin{conj}
Let $M$ be a discrete $d$-pseudomanifold. Then $\gamma_2(M)\geq 0$.
\end{conj}

\bigskip

 \noindent {\bf Acknowledgement:} The second author is supported by the institute fellowship at IIT Delhi. The third author is supported by the Ministry of Education of the Slovak Republic under grant VEGA 2/0056/25.

\smallskip

\noindent {\bf Data availability:} The authors declare that all data supporting the findings of this study are available within the article.

\smallskip

\noindent {\bf Declarations}

\smallskip

\noindent {\bf Conflict of interest:} No potential conflict of interest was reported by the authors.

\smallskip

 {
\end{document}